\documentclass[11pt]{amsart}

\usepackage[margin=1.5 in]{geometry}
\usepackage{amsmath, amssymb, amsthm}
\usepackage[pdftex]{color}
\usepackage{comment}
\usepackage[colorlinks,citecolor=blue,linkcolor=red]{hyperref}
\usepackage{xcolor}
\usepackage{marginnote}
\usepackage{hyperref}
\usepackage[all,arc]{xy}

\newtheorem{theorem}{Theorem}[section]
\newtheorem{lemma}[theorem]{Lemma}
\newtheorem{proposition}[theorem]{Proposition}
\newtheorem{corollary}[theorem]{Corollary}

\theoremstyle{remark}

\theoremstyle{definition}
\newtheorem{definition}[theorem]{Definition}

\newtheorem{remark}[theorem]{Remark}

\newcommand{\ev}{\mathrm{ev}}

\newcommand{\mc}{\mathcal}
\newcommand{\mb}{\mathbb}

\newcommand{\nc}{\newcommand}
\nc{\dmo}{\DeclareMathOperator}

\nc{\R}{\mathbb{R}}
\nc{\Z}{\mathbb{Z}}
\nc{\N}{\mathbb{N}}
\nc{\cS}{\mathcal{S}}
\nc{\iso}{\cong}
\dmo{\Diff}{Diff}
\dmo{\Homeo}{Homeo}
\dmo{\dist}{dist}
\dmo\BDiff{BDiff}
\dmo\SO{SO}
\dmo\slide{sl}
\dmo\im{im}
\dmo\id{id}
\dmo\Fix{Fix}
\dmo\Out{Out}
\dmo{\T}{\mathcal{T}}
\dmo{\Te}{\mathcal{T}^{\epsilon}}
\dmo{\Me}{\mathcal{M}^{\epsilon}}

\begin{document}

\title[Central limit theorems for counting measures]{Central limit theorems for counting measures in coarse negative curvature}

\author[I. Gekhtman]{Ilya Gekhtman}
\address{Department of Mathematics\\ 
University of Toronto\\ 
40 St George St\\ 
Toronto, ON, Canada\\}
\email{\href{mailto:ilyagekh@gmail.com}{ilyagekh@gmail.com}}

\author[S.J. Taylor]{Samuel J. Taylor}
\address{Department of Mathematics\\ 
Temple University\\ 
1805 North Broad Street\ 
Philadelphia, PA 19122, U.S.A\\}
\email{\href{mailto:samuel.taylor@temple.edu}{samuel.taylor@temple.edu}}

\author[G. Tiozzo]{Giulio Tiozzo}
\address{Department of Mathematics\\ 
University of Toronto\\ 
40 St George St\\ 
Toronto, ON, Canada\\}
\email{\href{mailto:tiozzo@math.toronto.edu}{tiozzo@math.toronto.edu}}

\date{\today}

\begin{abstract}
We establish central limit theorems for an action of a group $G$ on a hyperbolic space $X$
with respect to the counting measure on a Cayley graph of $G$.
Our techniques allow us to remove the usual assumptions of properness and smoothness of the space, or cocompactness of the action.
We provide several applications which require our general framework, including to lengths of 
geodesics in geometrically finite manifolds and to intersection numbers with submanifolds.
\end{abstract}

\maketitle

\section{Introduction} \label{sec:intro}

The goal of this paper is to provide a novel approach to the central limit theorem on groups acting on hyperbolic spaces, for sampling 
with respect to the word length in the group. We shall replace the traditional approach based on thermodynamic formalism with techniques coming from the theory of random walks on groups. This allows us to establish new applications, including central limit theorems for lengths of geodesics in geometrically finite hyperbolic manifolds, for intersection numbers with submanifolds, and for homomorphisms between hyperbolic groups.

\subsection*{Motivation}

The distribution of lengths of closed orbits for smooth flows on manifolds has long been a topic of considerable interest.  
For instance, Sinai \cite{Sinai-CLT} and then Ratner \cite{Ratner} proved a central limit theorem (CLT) for the geodesic flow
on a hyperbolic manifold (see also Lalley \cite{Lalley}). 
One prominent technique, pioneered by Sinai \cite{Sinai-Gibbs}, Bowen \cite{Bowen}, Ruelle \cite{Ruelle}, Parry--Pollicott \cite{PP}, and others, uses Markov partitions to reduce the study of smooth flows to symbolic dynamics to which one can apply tools from thermodynamic formalism. This approach has been successful in a variety of settings, especially applied to Anosov flows and their generalizations.

\smallskip
More recently, there has been a renewed interest in statistical properties of geodesic length and other geometric quantities with respect to a different sampling, namely according to the \emph{counting measure}, i.e. uniform measure on spheres in a Cayley graph of a finitely generated group $G$. For instance, Pollicott--Sharp \cite{PS} considered the ratio between the word length and the geometric length, while CLTs have been established for quasimorphisms on free groups by Horsham-Sharp \cite{HS}
and on general hyperbolic groups by Calegari--Fujiwara \cite{CF} and Bj\"orklund--Hartnick \cite{BH}.

In \cite{GTT3} the authors, building on \cite{GTT}, settled a conjecture of Chas--Li--Maskit \cite{CLM} about the distribution of hyperbolic lengths of closed geodesics on compact surfaces when sampling with respect to word length. 
Further, a CLT and statistical laws have been established for cocompact, proper actions of hyperbolic groups on CAT$(-1)$ spaces by Cantrell \cite{Can}.

All of these results are based on a symbolic coding and thermodynamic formalism. Although these  techniques are quite powerful, they necessarily impose strong constraints on the actions of interest, usually requiring that the space $X$ is CAT$(-1)$ and that the action $G \curvearrowright X$ is proper cocompact. While this is the case in the classical setting, they are not satisfied for most actions on Gromov hyperbolic spaces. 

\smallskip
The goal of this paper is to provide a new approach to the central limit theorem on groups $G \curvearrowright X$ acting on hyperbolic spaces, which will allow us to consider in particular: 

\begin{enumerate}

\item groups $G$ which are not necessarily word hyperbolic; 

\item actions on spaces $(X, d)$ which are $\delta$-hyperbolic, but not necessarily CAT$(-1)$ or proper; 

\item group actions $G \curvearrowright X$ which need not  be convex cocompact or even proper;

\item observables $\phi : G \to \mathbb{R}$ which are not necessarily H\"older continuous, and are not quasimorphisms. 

\end{enumerate}

For the sake of concreteness, we will now present a version of our main theorem (Theorem \ref{T:main}) from which we will then derive several applications. Our discussion here will be a special case of the most general theorems (Theorems \ref{T:main-almostss}, \ref{T:CLT-tau}) which we will state and prove in Section \ref{sec:almost_semi}. 

\subsection*{Main results}
Let $G$ be a finitely generated group acting by isometries on a $\delta$-hyperbolic metric space $(X, d)$, and fix a finite generating set $S$. 
We require that the action is \emph{nonelementary} in the sense that there are two independent loxodromic elements. 

Let $S_n := \{ g \in G \ : \ \Vert g \Vert = n \}$ be the sphere of radius $n$ for the word metric with respect to $S$. We denote as $\mc N_\sigma$ the Gaussian measure $d \mc N_\sigma(t) = \frac{1}{\sqrt{2 \pi} \sigma} e^{-t^2/2\sigma^2}\ dt$ if $\sigma > 0$, and the Dirac mass at $0$ if $\sigma = 0$.  We require that $G$ admits a \emph{thick bicombing} for $S$ and we refer the reader to Section \ref{S:graph-structure} for definitions. We note here that these general conditions are satisfied in a variety of settings; for example, see the applications below and Lemma \ref{L:bicombing}.

\begin{theorem} \label{T:main}
Let $G$ be a group which admits a thick bicombing for the generating set $S$. 
Let $G \curvearrowright X$ be a nonelementary action by isometries on a $\delta$-hyperbolic space $(X, d)$, 
and let $o \in X$ be a base point. 

\begin{enumerate}
\item \textup{(CLT for displacement)}
Then there exists $\ell > 0$, $\sigma \geq 0$ such that for any $a < b$ we have
$$\lim_{n \to \infty} \frac{1}{\# S_n} \# \left\{ g \in S_n \ : \   \frac{d(o, g o) - n \ell}{\sqrt{n}}  \in [a, b] \right\} =  \int_a^b d \mc N_\sigma(t).$$

\item \textup{(CLT for translation length)} 
Moreover, if $\tau(g)$ denotes the translation length of $g$ on $X$, we also have for any $a < b$
$$\lim_{n \to \infty} \frac{1}{\# S_n} \# \left\{ g \in S_n \ : \   \frac{\tau(g) - n \ell}{\sqrt{n}}  \in [a, b] \right\} =  \int_a^b d \mc N_\sigma(t).$$

\item
Further, $\sigma = 0$ if and only if there exists a constant $C$ such that 
$$|d(o, go) - \ell \Vert g \Vert | \leq C$$
for any $g \in G$.
\end{enumerate}
\end{theorem}

We remark that as a consequence of $(3)$, if $\sigma=0$ then the translation length of any $g\in G$ with respect to its action on $X$ is a constant multiple of its translation length in the word metric. Moreover, if the action $G\curvearrowright  X$ is not proper, then $\sigma > 0$.

\smallskip

Another way to formulate (1) is to say that we have the convergence in distribution
\[
\frac{d(o, g o) - n \ell}{\sqrt{n}} \longrightarrow \mc N_\sigma, 
\]
hence from now on we will use the above notation as a shorthand.

\subsection*{Applications}\label{sec:apps}
There are a number of applications to the above theorems and we summarize a few of them here. For the proofs, see Section \ref{sec:Apps}.

\subsubsection*{Geometrically finite hyperbolic manifolds}
First, let us state an extension of our previous work on surfaces \cite{GTT3} to general hyperbolic manifolds, possibly with cusps.
If $M = \mb H^n / \Gamma$ is a hyperbolic manifold and $\gamma \in \Gamma = \pi_1(M)$, then we set $\ell(\gamma)$ to be the length of the geodesic freely homotopic to $\gamma$ unless $\gamma$ is peripheral (i.e. homotopic into a cusp), in which case we set $\ell(\gamma)=0$.

\begin{theorem} \label{th:length}
Suppose that $M$ is a geometrically finite hyperbolic manifold
 and let $S'$ be any generating set for $\pi_1 (M)$. Then
there is a finite generating set $S \supset S'$ and $\ell,\sigma>0$ such that
\[
\frac{\ell(\gamma)- n \ell}{\sqrt{n}} \longrightarrow  \mc N_\sigma, 
\]
where $\gamma$ is chosen uniformly at random in the sphere of radius $n$ 
with respect to $S$.

If moreover $\pi_1(M)$ is word hyperbolic, then we can take $S = S'$.
\end{theorem}

The statement includes the cases where $M$ is either finite volume or convex cocompact, and is new even when  
$M$ is a finite area surface.
We remark that when $M$ is either convex cocompact or a surface, the above theorem works for \emph{any} generating set $S$. 
In the convex cocompact case, the needed action $\pi_1(M) \curvearrowright \mathbb{H}^n$ is sufficiently tame so that the techniques of thermodynamics may be applicable (\cite{PS}, \cite{Can}). However, this is not the case when the manifold $M$ has cusps.

We note that Theorem \ref{th:length} further extends to manifolds of variable negative curvature, as long as the peripheral subgroups are virtually abelian, and the same proof applies.

\subsubsection*{Geometrically infinite 3-manifolds}

In the case of $3$-manifolds, the previous result can be strengthened further as follows. 

\begin{theorem}\label{th:3d}
Let $M$ be a hyperbolic $3$--manifold such that $\pi_1(M)$ is finitely generated and not virtually abelian. Suppose further that $M$ does not have any rank $2$ cusps. Then for \emph{any} finite generating set $S$ of $\pi_1(M)$, there are $\ell,\sigma>0$ such that
\[
\frac{\ell(\gamma)- n \ell}{\sqrt{n}} \longrightarrow  \mc N_\sigma, 
\]
where $\gamma$ is chosen uniformly at random in the sphere of radius $n$ with respect to $S$.

Moreover, if $M$ has rank $2$ cusps, the same statement holds after enlarging the generating set as in Theorem \ref{th:length}.
\end{theorem}

To the authors' knowledge, this is the first CLT for lengths of closed geodesics for possibly geometrically infinite $3$--manifolds.

\subsubsection*{Intersection numbers with a submanifold}
For our next application, the required actions are on locally infinite trees, which are nonproper hyperbolic spaces.

Let $M$ be a smooth orientable manifold and $\Sigma$ a smooth orientable codimension$-1$ submanifold which is $\pi_1$-injective on each component. 
We say $\Sigma$ is \emph{fiber-like} 
if for each boundary component of the cut manifold $M |\Sigma$ its induced subgroup in the fundamental group of the corresponding component of $M|\Sigma$ has index at most $2$.

For $\gamma \in \pi_1 (M) $, let $i(\gamma, \Sigma)$ denote the minimal intersection number of $\Sigma$ with loops in $M$ freely homotopic to $\gamma$. 

\begin{theorem} \label{th:intersection}
Suppose that $M$ is a closed orientable hyperbolic manifold and let $S$ be any generating set for $\pi_1 (M)$. Let $\Sigma$ be a smooth orientable codimension$-1$ submanifold that is $\pi_1$-injective but not fiber-like.
Then there are $\ell,\sigma >0$ such that 
\[
\frac{i(\gamma, \Sigma) -  \ell n}{\sqrt{n}} \longrightarrow \mc N_\sigma,
\]
where $\gamma$ is chosen uniformly at random in the sphere of radius $n$ with respect to $S$. 
\end{theorem}

The theorem is new even for surfaces; in that context, Chas--Lalley \cite{CL} proved a CLT for self-intersection numbers of curves with respect to word length. Following Theorem \ref{th:length}, a similar result could be formulated for more general hyperbolic manifolds.

\subsubsection*{Homomorphisms between hyperbolic groups} Our next application is to homomorphisms between hyperbolic groups. Interestingly, the condition for nonzero variance can be recast in terms of the induced Patterson--Sullivan measures.

\begin{theorem}\label{th:homo-PS}
Suppose that $\phi \colon G \to G'$ is a homomorphism between hyperbolic groups such that the image of $\phi$ is not virtually cyclic. For any fixed generating sets $S$ and $S'$ of $G$ and $G'$, respectively, there are $\ell >0$ and $\sigma \ge0$ such that  
\[
\frac{\Vert \phi(g) \Vert_{S'}-  \ell \Vert g \Vert_{S}}{\sqrt{\Vert g\Vert_S }} \longrightarrow \mc N_\sigma,
\]
for $g\in G$ chosen uniformly at random in the sphere of radius $n$ with respect to $S$. 

Moreover, $\sigma = 0$ (i.e. the Gaussian is degenerate) if and only if $\phi$ has finite kernel and the induced homeomorphism $\partial \phi \colon \partial G \to \partial G'$ pushes the Patterson--Sullivan measure class for $(G,S)$ to the Patterson--Sullivan measure class for $(\phi(G), S')$.
\end{theorem}

The above result generalizes \cite [Theorem 1.6]{Can}, who proved a CLT where $\phi$ is the abelianization homomorphism, 
which in turn generalizes work of Rivin \cite{Rivin} for free groups. This is also a generalization of Calegari-Fujiwara \cite[Corollary 4.27]{CF}. 

\subsubsection*{Hyperplanes crossed in right-angled Artin and Coxeter groups} 
Our final application is to a collection of groups that is not necessarily relatively hyperbolic. 

Suppose that $G$ is a right-angled Artin group or right-angled Coxeter group that is not a direct product. Let $V$ be its set of vertex generators. For each $v \in V$, define a function $\#_v \colon G \to \mathbb{Z}$ that counts the number of occurrences of $v^{\pm 1}$ in a shortest spelling of $g \in G$ with respect to $V$. Equivalently, $\#_v(g)$ is the number of hyperplanes labeled by $v$ separating $o$ and $go$ in the cube complex associated to $G$.

\begin{theorem} \label{th:raag}
For $G$ as above, there are $\ell,\sigma > 0$ such that for any vertex $v$,
\[
\frac{\#_v(g) -  \ell n}{\sqrt{n}} \longrightarrow \mc N_\sigma,
\]
where $g$ is chosen uniformly at random in the sphere of radius $n$ with respect to the vertex generators.
\end{theorem}

\smallskip

We conclude by noting that our methods are sufficiently general to apply beyond the case of `nonpositively curved' groups. 
Moreover, we do \emph{not} need to assume that our counting measures are associated to geodesic combings.
See Theorems \ref{T:main-almostss}, \ref{T:CLT-tau} for the most general result. 
For example, by using the standard graph structure associated to the language of geodesics for a free group, we obtain a CLT for nonbacktracking random walks on \emph{any} group with a nonelementary action on a hyperbolic space $X$.

\subsection*{From thermodynamics to random walks}

Most central limit theorems for counting measures established so far use a coding for geodesics with finite paths, and 
then apply classical results in thermodynamic formalism, like the existence and uniqueness of Gibbs measures for shifts of finite type. 
There, the observable is assumed to be H\"older continuous with respect to the standard metric on the shift space. 

In this paper, instead, we \emph{do not} assume any good geometric property on the action.
Let us recall that displacement is not a quasimorphism, is in general not \emph{weakly combable} (in the language of \cite{CF}) if the action is not convex cocompact, and it is not a \emph{H\"older weight function} in the sense of \cite[Proposition 1]{PS} if $X$ is not CAT$(-1)$. 
Thus, the observable need not be H\"older and the thermodynamic approach does not appear to work.  
We also do not use transfer operators or $\zeta$-functions (as in e.g. \cite{PP}, \cite{Can}).

\smallskip
Rather, our general strategy is as follows.

\begin{enumerate}

\item
We start with a \emph{graph structure}, i.e. a graph whose paths parameterize the group elements we want to count. 
We first consider a vertex $v$ of this graph, and consider a random walk on the semigroup $\Gamma_v$
of loops based at this vertex. Here, we apply the CLT for cocycles for groups acting on hyperbolic spaces, as devised by Benoist-Quint \cite{BQ-hyperbolic} and generalized by Horbez \cite{Horbez} to actions on nonproper spaces.

\item
Then, we consider the set of paths in a maximal component for the graph as a suspension on the space of loops at $v$, 
and we apply results of Melbourne-T\"or\"ok \cite{MT} to ``lift" the CLT to the suspended transformation. 
To be precise, we need to consider a skew product over the shift space. 

\item
Now, we note that a thick graph structure is \emph{almost semisimple}, hence there exists a power $p$ for which the transition matrix $M^p$ is 
semisimple. We use this to prove that the counting measure starting at an initial vertex converges to a convex combination 
of stationary measures for the Markov chains on the maximal components. 

\item
Using biautomaticity, we show that all the CLTs for all Markov chains have the same mean and variance. This implies 
a CLT for the counting measure on the set of paths starting at any vertex in a semisimple structure.

\item 
Finally, for a general thick structure of period $p$ we condition on the first prefix of length $r$; since all these distributions 
for the conditional measures converge to the same law (by (4) above), the CLT for the entire sequence holds.

\end{enumerate}

\subsection*{Acknowledgments}

Gekhtman is partially supported by NSERC. 
Taylor is partially supported by NSF grant DMS-1744551 and the Sloan Foundation.
Tiozzo is partially supported by NSERC and the Sloan Foundation.

\section{Background}

\subsection{Graph structures for countable groups} \label{S:graph-structure}

Given a countable group $G$, we define a \emph{graph structure} on $G$ as a triple $(\Gamma, v_0, \ev)$, where 
$\Gamma$ is finite, directed graph, $v_0$ is a vertex of $\Gamma$ which we call its \emph{initial vertex}, and $\ev \colon E(\Gamma) \to G$ is a map that labels the edges of $\Gamma$ with group elements. 
Given this data, we extend the map $\ev$ by defining for each finite path $g = g_1 \dots g_n$ the group element 
$\ev(g) = \ev(g_1) \dots \ev(g_n)$. To simplify notation, we will use 
$\overline{g} = \ev(g)$ to denote the group 
element associated to the path $g$. We denote as $\Vert g \Vert$ the length of the path $g$. Throughout the paper, we assume that the graph structure is \emph{proper} in the sense that for each group element there are at most finitely many paths in the graph that evaluate to it.

For a graph structure $\Gamma$, we define $\Omega$ to be the set of all infinite paths starting at any vertex of $\Gamma$ and $\sigma \colon \Omega \to \Omega$ to be
the shift map. Given a path $\omega = (g_1, \dots, g_n, \dots)$, we denote as $w_n := g_1 \dots g_n$ its prefix of length $n$.

We define two vertices $v_i, v_j$ to be \emph{equivalent} if there is a path from $v_i$ to $v_j$ and a path from $v_j$ to $v_i$, 
and the \emph{components} of $\Gamma$ as the equivalence classes for this relation. 

We will denote by $M$ the transition matrix for $\Gamma$. By Perron-Frobenius, $M$ has a real eigenvalue of largest modulus, 
which we will denote by $\lambda$. Moreover, such a matrix is \emph{almost semisimple} if for any eigenvalue of maximal modulus, its geometric and algebraic multiplicity agree. Furthermore, such a matrix is \emph{semisimple} if its only eigenvalue of maximal modulus is real positive. 
We call a graph structure \emph{(almost) semisimple} if its associated transition matrix is. 

Let $\Gamma$ be almost semisimple, and let $\lambda$ be the leading eigenvalue of $M$. Then we define a vertex $v$ to be of \emph{large growth} if 
$$\lim_{n \to \infty} \frac{1}{n} \log \# \{ \textup{paths of length }n\textup{ starting at }v \} = \lambda$$
and of \emph{small growth} otherwise (in which case, the limit above is $< \lambda$). 
Furthermore, a component $C$ is \emph{maximal} if 
$$\lim_{n \to \infty} \frac{1}{n} \log \# \{ \textup{paths of length }n\textup{ inside }C \} = \lambda.$$

As discussed in \cite{GTT2}, the global structure of $\Gamma$ is as follows: there is no path between maximal components and vertices of large growth are precisely the ones which have a path to a maximal component.

Given a vertex $v$, we denote as $\Gamma_v$ the \emph{loop semigroup} of $v$, i.e. the set of all finite paths from $v$ to itself. 
This is a semigroup under concatenation, and all its elements lie entirely in the component of $v$. We denote as $\overline{\Gamma}_v$ 
the image of $\Gamma_v$ in $G$ under the evaluation map.

\begin{definition}[Thick graph structure]
A graph structure $\Gamma$ is \emph{thick} if for any vertex $v$ in a maximal component, there exists a finite set $B \subseteq G$ such that 
$$G = B \cdot  \overline{\Gamma}_v \cdot B$$
where the equality is \emph{in the group} $G$.
\end{definition}

In what follows, we often make the evaluation map implicit in our notation. 
In particular, if $G$ acts on a metric space $(X, d)$, $o \in X$ is a base point, and $g$ is a finite path in $\Gamma$, we will often write $go$ to mean the point $\overline{g}o \in X$.

The next lemma summarizes some properties of thick graph structures that we will need in the sequel.

\begin{lemma}\label{lem:thick_implies}
A thick graph structure $\Gamma$ is almost semisimple. Moreover, if $G \curvearrowright X$ is a nonelementary action on a hyperbolic space, then the actions of both semigroups $\Gamma_v$ and $\Gamma_v^{-1}$ on $X$ are also nonelementary, for each vertex $v$ contained in a maximal component. 
\end{lemma}

\begin{proof}
If the transition matrix $M$ of $\Gamma$ is not almost semisimple, then $M$ has a Jordan block for an eigenvalue of modulus $\lambda$ of size $k\ge 2$ (see \cite[Section 2]{GTT}). In particular, the growth of closed paths in $\Gamma$ is at least a constant times $n^{k-1}\lambda^n$. However, this is impossible if $\Gamma$ is thick because in this case the growth of closed paths is no more than the growth of closed paths in $\Gamma_v$ for $v$ in a maximal component. This, in turn, is bounded by a constant times $\lambda^n$ since it is no more than the growth of closed paths in its maximal component.

The statement that the action of $\Gamma_v$ (and hence $\Gamma_v^{-1}$) on $X$ is nonelementary is proven in \cite[Proposition 6.3]{GTT2}.
\end{proof}

\subsubsection*{Bicombings}

For particular applications, it is also useful to define the notion of a geodesic graph structure. A graph structure $\Gamma$ is \emph{geodesic} if 
the length $\Vert g \Vert$ of any path $g$ is equal the word length of $\overline{g}$ in the subgroup generated by the edge labels, using edge labels as the (finite) generating set. A geodesic graph structure is called a \emph{geodesic combing} if, in addition, the evaluation map is a bijection from the set of finite paths starting at $v_0$ to the set of elements of $G$. We say that $\Gamma$ is a geodesic combing \emph{associated to} a finite generating set $S$ if, up to adding inverses, $S$ is the set of edge labels for the graph structure. In this case, $\Vert g \Vert$ is equal to the word length of $\overline{g}$ with respect to $S$. 

We will make use of the following notion of biautomatic. See, for example, \cite{Mosher} and \cite[Lemma 2.5.5]{word-pro}. First, fix a word metric $d_G$ on $G$.
For a finite path $g$ in $\Gamma$, we denote by $g(i)$ the length $i$ prefix of $g$ when $i \le \Vert g \Vert$ and set $g(i) = g$ otherwise. 

\begin{definition}[Biautomatic graph structure] \label{D:fellow-travel}
A graph structure $\Gamma$ for $G$ is \emph{biautomatic} if the following holds. 
For any finite set $B \subseteq G$ there exists $C \ge 0$ so that if 
$g$ and $h$ are finite length paths in $\Gamma$, 
and $\overline{g} = b_1 \overline{h} b_2$ in $G$, with $b_1, b_2 \in B$, then 
\[
d_G (\overline{g(i)}, b_1 \overline{h(i)} ) \le C,
\]
for all $i\ge 0$.
\end{definition}

\begin{definition}
A \emph{bicombing} for the generating set $S$ on a group $G$ is a geodesic combing whose graph structure $\Gamma$ is biautomatic, 
and is \emph{thick} if the graph structure is thick. 
\end{definition}
 
We emphasize that the geodesic condition is used in the applications of our main theorem (as in Theorem \ref{T:main}--\ref{th:raag}), 
but is not required in the proof of the most general results, Theorems \ref{T:main-almostss}, \ref{T:CLT-tau}. 

\subsection{Cocycles and horofunctions}

Let $(X, d)$ be a metric space, and let $o \in X$ be a base point. Given $z \in X$, we define the 
\emph{Busemann function} $\rho_z : X \to \mb R$ as 
$$\rho_z(x) := d(x, z) - d(o, z).$$
Thus, setting
$$\Phi(z) := \rho_z$$
defines a map $$\Phi : X \to \textup{Lip}^1_o(X)$$
where $\textup{Lip}^1_o(X)$ is the space of $1$-Lipschitz functions on $X$ which vanishes at $o$. 

We define the \emph{horofunction compactification} $\overline{X}^h$ as the closure of $\Phi(X)$ in $\textup{Lip}^1_o(X)$, 
with respect to the topology of pointwise convergence. Elements of $\overline{X}^h$ will be called \emph{horofunctions}.
We denote as $\overline{X}_\infty^h$ the space of \emph{infinite horofunctions}, i.e. the set of $h \in \overline{X}^h$ such that 
$\inf_{x \in X} h(x) = - \infty$.

For any $\xi \in \overline{X}^h$, the \emph{Busemann cocycle} is defined as 
\begin{align*}
\beta_{\xi}(x, y) &:= \lim_{z_n \to \xi} [ d(y, z_n) - d(x, z_n)] \\
&= h_\xi(y) - h_\xi(x),
\end{align*}
where $h_\xi$ is the horofunction associated to $\xi$.  This has the usual cocycle property $\beta_{\xi}(x, z) =\beta_{\xi}(x, y) +\beta_{\xi}(y, z)$.

\begin{remark}
Benoist-Quint \cite{BQ-hyperbolic} and Horbez \cite{Horbez} define $B \colon G \times \overline{X}^h \to \mathbb{R}$ by
\[
B(g, \xi) = h_\xi(g^{-1} o).
\]
To compare their definition with ours:
\[
B(g, \xi) = h_\xi(g^{-1} o) = \lim_{z_n \to \xi} [ d(g^{-1}o, z_n) - d(o, z_n) ] = \beta_\xi(o, g^{-1} o).
\]
\end{remark}

\section{CLT for random walks on the loop semigroups} \label{S:loop}

Let $\Gamma$ be a graph structure for $G$, and let $v$ be a vertex in a maximal component. Recall that a loop is \emph{prime} if it is not itself a product of nontrivial loops, and that prime loops generate $\Gamma_v$ as a semigroup. 

Given a probability measure $\mu$ on the set of edges of $\Gamma$, one defines the \emph{first return measure} $\mu_v$ on $\Gamma_v$ as follows: 
if $l = g_1 \dots g_n$ is a prime loop in $\Gamma_v$,  
then we set 
$$\mu_v(l) := \mu(g_1) \cdots \mu(g_n).$$
We set $\mu_v(l) = 0$ for all other loops. Note that inversion defines a map $\Gamma_v \to \Gamma^{-1}_v$ and we define the measure $\check \mu$ on $\Gamma^{-1}_v$ by
$\check \mu (l) = \mu(l^{-1})$. These measures push forward to measures on the group $G$ under the evaluation map. We say that $\mu_v$ is \emph{nondegenerate} if it gives positive measure to any prime loop of $\Gamma_v$. 

Let $\mc M$ be a metric space on which $G$ acts by homeomorphisms. A measure $\nu$ on $\mc M$ is $\mu$\emph{-stationary} if 
$\nu = \int_G g_\star\nu \ d\mu(g)$, and $\mu$\emph{-ergodic} if it cannot be written as a nontrivial convex combination of $\mu$-stationary measures. 

\subsection{Central limit theorems for cocycles}

Recall that a \emph{cocycle} is a function $\sigma \colon G \times \mc M \to \mathbb{R}$ such that 
$$\sigma(gh, x) = \sigma(g, h x) + \sigma(h, x), \qquad  \forall g, h \in G, \forall x \in \mc M.$$
A cocycle $\sigma \colon G \times \mc M \to \mathbb{R}$ has \emph{constant drift} $\lambda$ if there exists $\lambda \in \mathbb{R}$ such that
$$\int_G \sigma(g, x) \ d\mu(g) = \lambda$$
for any $x \in \mc M$. 
A cocycle $\sigma \colon G \times \mc M \to \mathbb{R}$ is \emph{centerable} if it can be written as
\begin{equation*}
\sigma(g, x) = \sigma_0(g, x) + \psi(x) - \psi(g \cdot x)
\end{equation*}
where $\sigma_0$ is a cocycle with constant drift and where $\psi \colon \mc M \to \mathbb{R}$ is a bounded, measurable function.
In this case, we say that $\sigma_0$ is the \emph{centering} of $\sigma$; note that 
$\lambda = \int_{G \times \mc M} \sigma(g, x) \ d\mu(g) d \nu(x)$ for any $\mu$-stationary $\nu$.
We say that the cocycle $\sigma$ has \emph{finite second moment} with respect to a measure $\mu$ on $G$ if
$$ \int_G \sup_{x\in \mc M} | \sigma(g,x)|^2 \ d\mu(g)< +\infty.$$

\smallskip
We now use the following CLT for centerable cocycles: as remarked in \cite[Remark 1.7]{Horbez}, 
the proof is exactly the same as the proof of \cite[Theorem 4.7]{BQ-hyperbolic}.

\begin{theorem}[Central limit theorem for cocycles] \label{T:sigma-CLT}
Let $G$ be a discrete group, $\mc M$ be a compact metrizable $G$-space and $\mu$ a probability measure on $G$.
Let $\nu$ be a $\mu$-ergodic, $\mu$-stationary probability measure on $\mc M$, and let $\mc M_0$ be a $G$-invariant subset of $\mc M$ of full 
$\nu$-measure. Let $\sigma \colon G \times \mc M_0 \to \R$ be a 
centerable cocycle with drift $\lambda$ and finite second moment. 
Then there exist $\sigma \geq 0$ such that for any continuous $F \colon \mathbb{R} \to \mathbb{R}$ with compact support, 
we have for $\nu$-a.e. $x \in \mc M$, 
$$\lim_{n \to \infty} \int_G F\left( \frac{\sigma(g, x) - n \lambda}{\sqrt{n}} \right) \ d \mu^{*n}(g) = \int_{\mathbb{R}} F(t) \ d \mc N_\sigma(t).$$
\end{theorem}

We now apply this result to the loop semigroup. 
Let $F_v$ be the group freely generated by the prime loops in $\Gamma_v$.  

Let $N\colon  \Gamma_v \to \mathbb{Z}$ be the semigroup homomorphism $N(g) := - \Vert g \Vert$, where $\Vert g \Vert$ is the length in $\Gamma$
of the loop $g$.
There is a natural inclusion $\Gamma_v \to F_v$ as a subsemigroup and we can extend the semigroup homomorphism above 
to a group homomorphism $N\colon F_v \to \mathbb{Z}$. 
Moreover, we also extend the natural semigroup homomorphism $\Gamma_v \to G$, induced by evaluation, to a group homomorphism 
$e \colon F_v \to G$. Now, using the homomorphism $e \colon F_v \to G$, the free group $F_v$ has a nonelementary action on $X$, and moreover $\mu_v^{*n}$ is supported on $\Gamma_v \subseteq F_v$ for all $n\ge1$.

Finally, for some $\ell \in \mb R$ to be specified below, we define $\eta \colon F_v \times \overline X^h \to \mathbb{R}$ as 
$$\eta(g, \xi) := \beta_{\xi}(o, g^{-1} o) - \ell N(g).$$

\begin{lemma} 
Suppose that the action of $\Gamma_v$ on $X$ is nonelementary and $\mu_v$ is nondegenerate. Then for any $\ell \in \mathbb{R}$, the 
restriction of $\eta : F_v \times \overline X^h \to \mathbb{R}$ to $F_v \times \overline{X}_\infty^h$  is a centerable cocycle. 
\end{lemma}

\begin{proof}
We have 
\begin{align*}
\eta(gh, \xi) & = \beta_\xi(o, h^{-1} g^{-1} o) - \ell N(gh) \\
& = \beta_\xi(o, h^{-1} o) + \beta_{\xi}(h^{-1} o, h^{-1} g^{-1} o)  - \ell N(g) - \ell N(h)  \\
& = \beta_\xi(o, h^{-1} o) + \beta_{h \xi}(o, g^{-1} o)  - \ell N(g) - \ell N(h) \\
& = \eta(h, \xi) + \eta(g, h\xi)
\end{align*}
hence $\eta$ is a cocycle. Moreover, by \cite[Proposition 1.5]{Horbez}, using \cite[Corollary 2.7]{Horbez} and \cite[Proposition 2.8]{Horbez}, the cocycle $B(g, \xi) = \beta_\xi(o, g^{-1}o)$ is centerable on $F_v \times \overline{X}_\infty^h$. 
Then, since $\eta(g, \xi) - B(g, \xi) = \ell N(g)$ is a homomorphism and depends only on $g$, we have that 
$\eta(g, \xi)$ is also centerable on $F_v \times \overline{X}_\infty^h$. 
\end{proof}

Thus, as a consequence of Theorem \ref{T:sigma-CLT}, we obtain the following. 

\begin{corollary} \label{C:CLT-loop}
Let $\Gamma$ be a thick structure, let $v$ be a vertex in a maximal component of $\Gamma$. Suppose that the first return measure $\mu_v$ 
is nondegenerate, and let $\nu_v$ 
be a $\check{\mu}_v$-ergodic, $\check{\mu}_v$-stationary measure on $\overline{X}^h$. Then there exist $\ell, \sigma \ge 0$ such that for any continuous $F \colon \mathbb{R} \to \mathbb{R}$ with compact support, we have
for $\nu_v$-a.e. $\xi$, 
$$\lim_{n \to \infty} \int_G F \left( \frac{\beta_\xi(o, g o) - \ell \Vert g \Vert}{\sqrt{n}} \right) \ d \mu_v^{*n}(g) = \int_{\mathbb{R}} F(t) \ d \mc N_\sigma(t).$$
\end{corollary}

\begin{proof}
We apply Theorem \ref{T:sigma-CLT} to the measure $\check{\mu}_v$, supported on $\Gamma_v^{-1}$, where $\ell$ is chosen so that 
$\lambda = \int_{F_v \times \overline X^h} \eta(g, \xi) \ d\check{\mu}_v(g) d\nu_v(\xi) = 0$. Note that by \cite[Proposition 4.4]{MT} and the fact that $\Gamma_v^{-1}$ is nonelementary, we have $\nu_v(\overline{X}^h_\infty) = 1$. 
Moreover, for any $g \in \Gamma_v$ we have 
\[
\eta(g^{-1}, \xi) = \beta_{\xi}(o, g o) - \ell \Vert g \Vert.
\qedhere
\]
\end{proof}

\subsection{Skew products and invariance on the loop semigroup}

Let $\mc M$ be a compact metric space with a continuous $G$-action. We define the \emph{skew product } $T\colon \Omega \times \mc M \to \Omega \times \mc M$ as 
$$T(\omega, \xi) := (\sigma(\omega), g_1^{-1} \xi)$$
where $\omega = (g_1, g_2, \dots)$. 

A graph structure $\Gamma$ is \emph{primitive} if its associated transition matrix $M$ is primitive, i.e. has a positive power. 
Now let $\Gamma$ be a primitive graph structure, let $v$ be a vertex of $\Gamma$, let $\Gamma_v$ be the loop semigroup, and let $\mu_v$ be the first return measure. 

Finally, let $\Omega_v = (\Gamma_v)^\mathbb{N}$ with shift map $\sigma_v$. 
To highlight the difference, we denote the elements of $\Gamma_v^\mathbb{N}$ as $(l_1, l_2, \dots)$,
since each element of the sequence is a loop, while the elements of $\Omega$ will be denoted as $\omega = (g_1, g_2, \dots)$, 
since its elements are edges.
Let us define the map $T_v : \Omega_v \times \mc M \to \Omega_v \times \mc M$ as 
$$T_v(\omega, \xi) = (\sigma_v (\omega), l_1^{-1} \xi).$$

\begin{lemma} \label{L:invariant1}
A measure $\nu$ on $\mc M$ is $\check{\mu}_v$-stationary  if and only if $\mu_v^\mathbb{N} \otimes \nu$ is $T_v$-invariant. 
\end{lemma}

\begin{proof} 
Fix $C \subset \Omega_v$ measurable and let $C_l \subset \Omega_v$ be the subset consisting of sequences beginning with $l \in \Gamma_v$ such that 
$\sigma_v (C_l) = C$. Then for any $A \subset \mc M$ measurable, 
\[
T_v^{-1} (C \times A) = \bigcup_l C_l \times l A.
\]
Since
\begin{align*}
\mu_v^\mathbb{N} \otimes \nu \left (\bigcup_l C_l \times l A \right) 
&= \mu(C) \sum_l \mu(l) \nu(lA) \\
&= \mu(C) \sum_l \check \mu(l) l_*\nu(A) ,
\end{align*}
the lemma follows.
\end{proof}

\begin{lemma} \label{lem:measure}
There exists an ergodic $\check{\mu}_v$-stationary measure $\nu_v$ on $\mc M$ such that the product measure $\mu_v^\mathbb{N} \otimes \nu_v$
is $T_v$-invariant and ergodic. 
\end{lemma}

\begin{proof}
Since $\mc M$ is a compact metric space, there exists a $\check{\mu}_v$-stationary measure $\nu_1$ on $\mc M$. 
Then by Lemma \ref{L:invariant1} the measure $\lambda_1 := \mu_v^\mathbb{N} \otimes \nu_1$ is $T_v$-invariant. 
If $\lambda_1$ is not ergodic, let us consider its ergodic decomposition, and take one of its ergodic components $\lambda_v$. 
By definition, $\lambda_v \ll \lambda_1$ and $\lambda_v$ is $T_v$-invariant and ergodic. 
Then by \cite[Corollary 3.1]{Morita}, $\lambda_v$ is of the form $\lambda_v = \mu_v^\mathbb{N} \otimes \nu_v$ 
for some measure $\nu_v$ on $\mc M$. Finally, again by Lemma \ref{L:invariant1}, the measure $\nu_v$ is $\check{\mu}_v$-stationary. 
\end{proof}

\begin{lemma} \label{lem:sum_bus}
Consider the function $f\colon \Omega \times \overline{X}^h \to \mathbb{R}$ defined as 
$$f(\omega, \xi) := \beta_\xi(o, g_1 o).$$
Then for any $n$ we have 
\begin{equation} \label{E:skew}
\sum_{j = 0}^{n-1} f(T^j(\omega, \xi)) = \beta_{\xi}(o, w_n o).
\end{equation}
\end{lemma}

\begin{proof}
The cocycle property implies 
\[
\beta_{\xi} (o, w_n o) = \sum_{j =0}^{n-1} \beta_{\xi} (w_j o, w_{j+1} o) = \sum_{j = 0}^{n-1} \beta_{w_j^{-1} \xi}(o, g_{j+1} o)
\]
for any $\xi \in \overline{X}^h$. 
Moreover, by definition and $G$-equivariance we have 
$$f(T^j(\omega, \xi)) = \beta_{w_j^{-1} \xi}(o, g_{j+1} o)$$
and the claim follows. 
\end{proof}

An analogous statement holds by replacing $T, \Omega$ by $T_v, \Omega_v$. 

\section{CLT for Markov chains of primitive graph structures}
\label{sec:clt_primitve}

We begin by recalling the following: If $\Gamma$ is a directed graph whose transition matrix $M$ is primitive with leading eigenvalue $\lambda$, then 
\[
\lim_{n \to \infty} \frac{M^n}{\lambda^n} = \rho u^T,
\]
where $M \rho = \lambda \rho$, $u^T M = \lambda u^T$, and $u^T \rho = 1$. Then the stationary measure for the corresponding Markov chain is given by setting the starting probability at vertex $v_i$ as $\pi_i = \rho_i u_i$ and assigning 
to an edge from $v_i$ to $v_j$ the transition probability $p_{ij} = \frac{\rho_j}{\lambda \rho_i}$. 
This gives the measure of maximal entropy $\mathbb{P}$ for the path space $\Omega$ of $\Gamma$, also called the \emph{Parry measure} \cite{Parry}.
For a vertex $v$ of $\Gamma$, we use $\mathbb{P}_v$ to denote the measure on the space of paths $\Omega_v \subset \Omega$ starting at $v$ obtained by beginning the Markov chain at $v$ and using the above transition probabilities. From now on we use these edge probabilities to define 
 the first return measure $\mu_v$ as in Section \ref{S:loop}. 

In this section we prove the following result.

\begin{theorem} \label{T:Markov}
Suppose that $\Gamma$ is a primitive graph structure and let $\mu_n$ be the $n$-th step distribution of the Markov chain on $\Gamma$. There are constants $\ell$ and $\sigma$ such that
 for any continuous function $F \colon \mathbb{R} \to \mathbb{R}$ with compact support, we have 
\[
\lim_{n \to \infty} \int_G F \left( \frac{d(o, go) - n \ell}{\sqrt{n}} \right) \ d \mu_n(g) = \int_{\mathbb{R}} F(t) \ d \mc N_{\sigma}(t). 
\]
\end{theorem}

The main technique to obtain the CLT for the Markov chain as above from the one from the random walk on the loop semigroup 
is using a \emph{suspension flow}, adapting the approach of Melbourne-T\"or\"ok \cite{Mel-Tor} for dynamical systems.

\subsection{Suspension flows}

Let $S \colon (\mathcal{X}, \lambda) \to (\mathcal{X}, \lambda)$ be a measure-preserving dynamical system, and let $r \colon \mathcal{X} \to \mathbb{N}$
be a measurable, integrable function, which we call the \emph{roof function}. Then the \emph{discrete suspension flow} of $S$ with roof function $r$ is the dynamical 
system given by the map $\widehat{S} : \widehat{\mathcal{X}} \to \widehat{\mathcal{X}}$ where 
$$\widehat{\mathcal{X}} := \{ (x, n) \in \mathcal{X} \times \mathbb{N} \ : \ 0 \leq n \leq r(x) - 1\}$$
with measure $\widehat{\lambda} := \frac{1}{\overline{r}} \left( \lambda \otimes \delta \right)$, where $\delta$ is the counting measure on $\mathbb{N}$ and 
$\overline{r} := \int_{\mathcal{X}} r \ d\lambda$.
Then, the map $\widehat{S}$ is defined as 
$$\widehat{S}(x, n) = \left \{ \begin{array}{ll} 
(x, n+1) & \textup{if }n \leq r(x) - 2 \\
(S(x), 0) & \textup{if }n = r(x) - 1.
\end{array} \right.$$

Since in this case the system has discrete time, the above construction is also called a \emph{Kakutani skyscraper}.

The main theorem of Melbourne-T\"orok \cite[Theorem 1.1]{Mel-Tor} is the following.

\begin{theorem}  \label{T:suspend}
Let $S\colon (\mathcal{X}, \lambda) \to (\mathcal{X}, \lambda)$ be an ergodic, measure-preserving transformation, and let 
$\widehat{S} \colon (\widehat{\mathcal{X}}, \widehat{\lambda}) \to  (\widehat{\mathcal{X}}, \widehat{\lambda})$ be the suspension flow with roof function $r$.
Let $\phi \colon \widehat{\mathcal{X}} \to \mathbb{R}$ be such that $\int \phi \ d\widehat{\lambda} = 0$, and define $\Phi(x) := \sum_{k = 0}^{r(x) -1} \phi(x, k)$. 
Let $\phi \in L^b(\widehat{\mathcal{X}})$ and let $r \in L^a(\mathcal{X})$ be the roof function, with $(1 - 1/a)(1-1/b) \geq 1/2$.
Suppose that $\Phi$ and $r$ satisfy a CLT. Then $\phi$ satisfies a CLT. 

Moreover, if the CLT for $\Phi$ has variance $\sigma_1^2$, then the CLT for $\phi$ has variance $\sigma^2 = \frac{\sigma_1^2}{\overline{r}}$.
\end{theorem}

\subsection{Invariant measure on the suspended space}
For any $\omega \in \Omega_v$, let $r(\omega)$ be the length in $\Gamma$ of $l_1(\omega)$. 
This is the first return time for the loop determined by $\omega$.
Let us define the suspension of the skew product
$$\Omega^{(s)} := \{ (\omega, k, \xi) \in \Omega_v \times \mathbb{N} \times \mc M \ : \ 0 \leq k \leq r(\omega) - 1\}$$
and 
$$\widehat{T}(\omega, k , \xi)  = \left\{ \begin{array}{ll} 
(\omega, k+1, \xi) & \textup{if }k \leq r(\omega) - 2 \\
(\sigma_v(\omega), 0, l^{-1}_1 \xi) & \textup{if } k = r(\omega) - 1.
\end{array} \right.$$
Let us now denote $R := \int_{\Gamma_v} \Vert g \Vert \ d \mu_v(g) = \int r(\omega) \ d \mb P_v(\omega)$ and define the probability measure $\nu^{(s)} := \frac{1}{R}\left( \mu_v^\mathbb{N} \otimes \delta \otimes \nu_v \right)$ on $\Omega^{(s)}$.

\begin{lemma}
Let $\nu_v$ be $\check{\mu}_v$-stationary measure constructed in Lemma \ref{lem:measure}. Then $\nu^{(s)}$ 
on $\Omega^{(s)}$ is $\widehat{T}$-invariant and ergodic.
\end{lemma}

\begin{proof}
It suffices to check invariance of the measure using cylinder sets $C_{l_1, \dots, l_n}$ consisting of loops beginning with $l_1\ldots l_n$.
We have
$$\widehat{T}^{-1}(C_{l_1, \dots, l_n} \times \{ k \} \times A) = \left\{ \begin{array}{ll} 
C_{l_1, \dots, l_n} \times \{ k-1 \} \times A & \textup{if }k > 0\\
\bigsqcup_{l \in P_v}  C_{l, l_1, \dots, l_n} \times \{ \Vert l \Vert -1 \} \times l A & \textup{if }k = 0\\
\end{array}
\right.$$
where $P_v \subseteq \Gamma_v$ is the set of prime loops. Hence in the first case, the equality 
$$\nu^{(s)}(\widehat{T}^{-1}(C_{l_1, \dots, l_n} \times \{ k \} \times A))  =  \nu^{(s)}(C_{l_1, \dots, l_n} \times \{ k \} \times A)$$ 
is obvious. In the second case, 
\begin{align*}
\nu^{(s)}(\widehat{T}^{-1}(C_{l_1, \dots, l_n} \times \{ k \} \times A)) & = \frac{1}{R} \sum_{l \in P_v} \mu_v(l) \mu_v(l_1)\dots \mu_v(l_n) \nu_v(lA) \\
& =  \frac{1}{R} \mu_v(l_1)\dots \mu_v(l_n) \sum_{l \in P_v} \mu_v(l) \nu_v(lA) \\
& = \frac{1}{R} \mu_v(l_1)\dots \mu_v(l_n) \nu_v(A) \\
& = \nu^{(s)}(C_{l_1, \dots, l_n} \times \{ k \} \times A)
\end{align*}
hence $\nu^{(s)}$ is $\widehat{T}$-invariant. 
Moreover, the suspension of an ergodic measure is ergodic, see e.g. \cite[Proposition 1.11]{Sarig}. 
\end{proof}

\subsection{Pushforward of the $\widehat{T}$-invariant measure to $\Omega \times \mc M$}

Recall that $\Omega$ is the space of all infinite sample paths in $\Gamma$ starting at any vertex. 
Let us define the projection $\pi : \Omega^{(s)} \to \Omega \times \mc M$ as 
$$\pi(\omega, k, \xi) = (\sigma^k(\omega), (g_1\dots g_k)^{-1} \xi)$$
and recall the skew product $T: \Omega \times \mc M \to \Omega \times \mc M$ is 
$$T(\omega, \xi) := (\sigma(\omega), g_1^{-1} \xi).$$ 

\begin{lemma} \label{lem:commute}
The following diagram commutes: 
$$\xymatrix{(\Omega^{(s)}, \nu^{(s)}) \ar@(ul,ur)^{\widehat{T}} \ar^{\pi}[r] \ar[d] & \Omega \times \mc M \ar@(ul,ur)^{T} \\
(\Omega_v \times \mc M, \mu_v^\mathbb{N} \otimes \nu_v)}$$
As a consequence, in the hypotheses of the previous lemmas, the measure $\overline{\nu} := \pi_\star \nu^{(s)}$ is $T$-invariant and ergodic. 
\end{lemma}

\begin{proof}
We show that the horizontal arrow is equivariant for the shifts. This follows from the fact that if we write $l_1(\omega)$ for the first return loop of $\omega$ then $l_1(\omega) = g_1(\omega)\ldots g_{r(\omega)}(\omega)$. Hence,  
\begin{align*}
\pi \circ \widehat T (\omega, r(\omega) -1, \xi) &=  \pi ((\sigma_v(\omega),0, l_1^{-1}\xi))\\
&=  (\sigma^{r(\omega)}(\omega), (g_1\ldots g_{r(\omega)})^{-1}\xi),
\end{align*}
which is equal to $T \circ \pi ((\omega, r(\omega) -1, \xi))$. The other cases being trivial, this proves the first statement. 

Finally, since $\overline \nu$ is the pushforward of an ergodic measure, it is ergodic.
\end{proof}

\subsection{Return times and invariant measures for the Markov chain}

Recall that in the previous section we produced a measure $\nu_v$ on $\mc M$ which is $\check{\mu}_v$-stationary and such 
that the product measure $\mu_v^\mathbb{N} \otimes \nu_v$ is $T_v$-invariant and ergodic. Then, by lifting it to the suspension and 
pushing it forward to $\Omega \times \mc M$, we have an ergodic, $T$-invariant measure $\overline{\nu}$ on $\Omega \times \mc M$.

Now, for any vertex $w$ other than $v$ we define the measure $\nu_w$ on $\mc M$ as 
\begin{equation}
\nu_w(A) = \sum_{\gamma \in \Gamma_{v, w}} \mu(\gamma) \nu_v(\gamma A)
\end{equation}
where the sum is over the set $\Gamma_{v, w}$ of all paths $\gamma$ from $v$ to $w$ which do not pass through $v$ in their middle, 
and $\mu(\gamma)$ is the product of the measures of the edges of $\gamma$. Recall also we denote as $\mathbb{P}_w$ 
the Markov chain measure on the space of infinite sample paths starting at $w$.

\begin{lemma}
We have 
$$\overline{\nu} := \frac{1}{R} \sum_w \mathbb{P}_w \otimes \nu_w.$$
\end{lemma}

\begin{proof}
Let $w$ be a vertex, and let $g_1, g_2, \dots, g_n$ be a finite path starting from $w$. We have for any measurable $A \subseteq \mc M$
$$\pi^{-1}(C_{g_1, \dots, g_n} \times A) = \left\{ \begin{array}{ll} 
C_{g_1, \dots, g_n} \times \{0 \} \times A & \textup{if }w = v \\
\bigsqcup_{\gamma \in \Gamma_{v, w}} C_{\gamma, g_1, \dots, g_n} \times \{|\gamma|\} \times \gamma A & \textup{if }w \neq v \\
\end{array}\right.$$
where the union is over the set $\Gamma_{v, w}$ of all paths $\gamma$ from $v$ to $w = s(g_1)$ which do not pass through $v$ in their middle. 
Thus we have 
\begin{align*}
\overline{\nu}(C_{g_1, \dots, g_n} \times A) &  = \frac{1}{R} \sum_{\gamma \in \Gamma_{v, w}} \mu(\gamma) \mu(g_1) \dots \mu(g_n) \nu_v(\gamma A) \\
& = \frac{1}{R} \mu(g_1) \dots \mu(g_n) \nu_w( A) \\ 
& = \frac{1}{R} \mathbb{P}_w(C_{g_1, \dots, g_n}) \nu_w(A)
\end{align*}
which proves the claim, since both measures agree on all rectangles. 
\end{proof}

Recall that $R = \int r(\omega) \ d\mathbb{P}_v(\omega)$, and set $n_w = \nu_w(\mc M)$.
Here we show

\begin{lemma} \label{lem:better_work}
We have the identities:
\begin{enumerate}
\item
$R = \frac{1}{\pi_v}$
\item
$\pi_w = \frac{n_w}{R}$ for any vertex $w$ of $\Gamma$.
\end{enumerate}
\end{lemma}

Note that if we replace $\Gamma$ with the graph $\overline \Gamma$ obtained by reversing the direction of each edge, then the transition matrix for $\overline \Gamma$ is $M^T$ and so we have that $\rho$ and $u$ switch roles. In particular, new transition probabilities on edges from $v_i$ to $v_j$ are (in terms of the quantities defined in Section \ref{sec:clt_primitve}) $p_{ij} = \frac{u_j}{\lambda u_i}$ but the stationary measure on vertices is unchanged. 

\begin{proof}[Proof of Lemma \ref{lem:better_work}]
(1) is well-known. 
To prove (2), recall that $\Gamma_{v, w}$ is the set of all paths $\gamma$ from $v$ to $w$ which do not pass through $v$ in their middle. Hence, if we reverse all the paths in this set, we obtain $\overline \Gamma_{v,w}$ the set of all paths $\overline \gamma$ from $w$ to $v$ which do not pass through $v$ in their middle. Note that since almost every path staring at $w$ passes through $v$
\begin{align*}
1&= \sum_{\overline \gamma \in \overline \Gamma_{v,w}} \overline \mu(\overline \gamma)\\
&= \frac{u_v}{u_w} \sum_{\overline \gamma \in \overline \Gamma_{v,w}} \lambda^{-|\overline \gamma|},
\end{align*}
where $\overline \mu(\overline \gamma)$ is the product of the measures of the edges of $\overline \gamma$ with respect to $\overline p$ and we have used our previous observation about $\overline p$.

Using this and the fact that 
\[
 \sum_{\overline \gamma \in \overline \Gamma_{v,w}} \lambda^{-|\overline \gamma|} =  \sum_{ \gamma \in  \Gamma_{v,w}} \lambda^{-| \gamma|},
\]
we compute,
\begin{align*}
n_w &= \nu_w(\mc M) \\
&= \sum_{\gamma \in \Gamma_{v, w}} \mu(\gamma) 
= \frac{\rho_w}{\rho_v}  \sum_{\gamma \in \Gamma_{v, w}} \lambda^{-|\gamma|}\\
&= \frac{\rho_w}{\rho_v} \cdot \frac{u_w}{u_v}
= \frac{\pi_w}{\pi_v}.
\end{align*}
Hence, the lemma follows from (1). 
\end{proof}

\subsection{The Central Limit Theorem for the Markov chain}

We are now in a position to prove Theorem \ref{T:Markov}. 
By Melbourne-T\"or\"ok (\cite{Mel-Tor}, Theorem 1.1), we have:

\begin{proposition} \label{P:MT1}
Let $\phi \colon  \Omega \times \overline{X}^h \to \mathbb{R}$ belong to $L^b(\Omega \times \overline{X}^h, \overline{\nu})$ for some $b > 2$, and let 
$m := \int \phi \ d\overline{\nu}$. 
Define $\Phi \colon \Omega_v \times \overline X^h \to \mb R$ as $\Phi(\omega, \xi) := \sum_{k = 0}^{r(\omega) -1} 
\phi(T^k(\omega,\xi)) - m  r(\omega)$, 
and suppose that 
$$\frac{\sum_{j = 0}^{n-1} \Phi \circ T_v^j }{\sqrt{n}}$$
converges to a normal distribution in probability on $(\Omega_v \times \overline{X}^h, \mu_v^{\mb N} \otimes \nu_v)$.
Then the sequence
$$ \frac{\sum_{j = 0}^{n-1} \phi \circ T^j - n m}{\sqrt{n}} $$
converges to a normal distribution in probability on $(\Omega \times \overline{X}^h, \overline{\nu})$.
\end{proposition}

\begin{proof}
Note that since $r$ has exponential tail, it belongs to $L^a(\Omega_v)$ for any $a \geq 1$. Then the condition $(1-1/a)(1-1/b) \geq 1/2$ is satisfied as long 
as $b > 2$. 
Moreover, $(r \circ T_v^n(\omega))_{n}$ is a sequence of independent random variables and so it satisfies a CLT.
Hence, we can apply Theorem \ref{T:suspend} to obtain a central limit theorem for the observable $\phi \circ \pi - m$ 
and the system $\widehat{T}$, with measure $\nu^{(s)}$. 
Moreover, since $\phi \circ \pi \circ \widehat{T}^n = \phi \circ T^n \circ \pi$ by Lemma \ref{lem:commute}, 
this is equivalent to a central limit theorem for the observable $\phi$ on 
the system $T$ with the measure $\pi_*(\nu^{(s)}) = \overline{\nu}$.
\end{proof}

\begin{proposition} \label{P:MT2}
There exist $\ell, \sigma$ such that for any continuous, compactly supported $F \colon \mathbb{R} \to \mb R$ one has
$$\int_{\Omega \times \overline{X}^h}   F\left(  \frac{\beta_\xi(o, g_1\dots g_n o) - n \ell}{\sqrt{n}}  \right) d \overline{\nu}(\omega, \xi)  \to \int_{\mb R} F(t) \ d \mc N_\sigma(t),$$
as $n \to \infty$.
\end{proposition}

\begin{proof}
Let us apply the previous Proposition with $\phi = f$ where $f \colon \Omega \times \overline X^h \to \mathbb{R}$ is defined as $f(\omega, \xi) := \beta_\xi(o, g_1 o)$.
Then by definition of $\Phi$ and $f$, Lemma \ref{lem:sum_bus} gives that for every $\omega \in \Omega_v$
\begin{align*} 
\Phi(\omega, \xi) & = \sum_{k = 0}^{r(\omega) - 1} f(T^k(\omega,\xi)) - \ell  r(\omega)\\
&= \beta_{\xi}(o, w_{r(\omega)} o)  - \ell r(\omega) =: f_v(\omega, \xi),
\end{align*}
where $\ell = m = \frac{\int \beta_\xi(o, go) \ d\mu_v(g) d\nu_v(\xi)}{\int \Vert g \Vert \ d\mu_v(g)}$.
Now, by Corollary \ref{C:CLT-loop}, integrating in $d\nu_v$ 
we have for some $\sigma_1 \geq 0$ 
$$\int_{G \times \overline{X}^h}  F\left(  \frac{\beta_\xi(o, go) - \ell \Vert g \Vert}{\sqrt{n}}  \right) d\mu_v^{*n}(g) d\nu_v(\xi)   \to \int_{\mb R} F(t) \ d \mc N_{\sigma_1}(t).$$
Note moreover that $\beta_\xi(o, l_1 \dots l_n o) - \ell \Vert l_1 \dots l_n \Vert = \sum_{j = 0}^{n-1} f_v (T_v^j(\omega, \xi))$, 
hence we can rewrite the above equation as 
$$\int_{\Omega_v \times \overline{X}^h}  F\left(  \frac{\sum_{j = 0}^{n-1} f_v( T_v^j(\omega, \xi))}{\sqrt{n}}  \right) d(\mu_v^{\mb N} \otimes \nu_v)(\omega, \xi)   \to \int_{\mb R} F(t) \ d \mc N_{\sigma_1}(t).$$
Thus, by Proposition \ref{P:MT1} and the above calculation, we also have (for some different $\sigma$)
$$\int_{\Omega_v \times \overline{X}^h}  F\left(  \frac{\sum_{j = 0}^{n-1} f \circ T^j(\omega, \xi) - n \ell}{\sqrt{n}}  \right) d \overline{\nu}(\omega, \xi)   \to \int_{\mb R} F(t) \ d \mc N_\sigma(t).$$
The claim follows by again using that by Lemma \ref{lem:sum_bus},
we have
$\sum_{j = 0}^{n-1} f \circ T^j(\omega, \xi) = \beta_\xi(o, g_1 \dots g_n o)$.
\end{proof}

Now, we will need to go from the CLT for the Busemann cocycle to the one for displacement. To do so, we use the following variation of 
\cite[Proposition 3.3]{BQ-hyperbolic}. 

By \cite[Proposition 4.4]{MT} and the fact that $\Gamma_v^{-1}$ is nonelementary, we have $\nu_v(\overline{X}^h_\infty) = 1$ for any vertex $v$. 

\begin{lemma} \label{L:indep}
For any $\epsilon > 0$ there exists $T$ such that for all vertices $w$ in $\Gamma$, all $\xi \in \overline{X}^h_\infty$ and all $n \geq 1$ we have
\[
\mathbb{P}_w \left(\omega \ : \  |d(o, g_1 \dots g_n o) - \beta_\xi(o, g_1 \dots g_n o) | \leq T \right) \geq 1 - \epsilon.
\]
\end{lemma}

\begin{proof}
Recall that by \cite[Section 3.3]{MT} there exists a $G$-equivariant map $\pi \colon \overline{X}^h_\infty \to \partial X$, where $\partial X$ is the Gromov boundary.
Then, by definition of Gromov product and $\delta$-hyperbolicity, we have 
\begin{equation} \label{E:gproduct}
d(o, go) - \beta_{\xi}(o, go) = 2 (go, \pi(\xi))_o + O(\delta)
\end{equation}
for any $\xi \in \overline{X}^h_\infty$.
Now, since the pushforward of the stationary measure $\mathbb{P}_w$ for the Markov chain starting at $w$ 
to the Gromov boundary of $X$ is not atomic (\cite[Lemma 4.2]{GTT2}), we have that for every $\epsilon > 0$ there exists $T$ such that 
$$\mathbb{P}_w({ \omega \in \Omega_w \ : \ \sup_{n \geq 1} (w_n o, \pi(\xi))_o \leq T }) \geq 1 - \epsilon$$
for all $\xi \in \overline{X}^h_\infty$ and for all $w$. This, combined with eq. \eqref{E:gproduct}, yields the desired estimate.
\end{proof}

\begin{proof}[Proof of Theorem \ref{T:Markov}]
Let $F \colon \mathbb{R} \to \mathbb{R}$ be continuous with compact support. Since $F$ is uniformly continuous and by Lemma \ref{L:indep}, for any $\eta > 0$ there exists 
$n_0$ such that for any $n \geq n_0$, any $w$ and any $\xi \in \overline{X}^h_\infty$ one has
$$\left| F\left( \frac{d(o, w_n o) - n \ell}{\sqrt{n}} \right) - F\left( \frac{\beta_\xi(o, w_n o) - n \ell}{\sqrt{n}} \right) \right| < \eta$$
with probability $\mathbb{P}_w$ at least $1 - \epsilon$. Thus, since $\overline{\nu} = \frac{1}{R} \sum_w \mathbb{P}_w \otimes \nu_w$,
for any $\eta > 0$ there exists $n_0$ such that for any $n \geq n_0$ we have 
\begin{equation} \label{E:eta}
\left| F\left( \frac{d(o, g_1 \dots g_n o) - n \ell}{\sqrt{n}} \right) - F\left( \frac{\beta_\xi(o, g_1 \dots g_n o) - n \ell}{\sqrt{n}} \right) \right| < \eta
\end{equation}
on a subset of $\Omega \times \overline{X}^h$ of $\overline{\nu}$-measure $\geq 1 - \epsilon$.
On the other hand, by Proposition \ref{P:MT2}, we have
$$\int_{\Omega \times \overline{X}^h}   F\left(  \frac{\beta_\xi(o, g_1\dots g_n o) - n \ell}{\sqrt{n}}  \right) d \overline{\nu}(\omega, \xi)  \to \int_{\mb R} F(t) \ d \mc N_\sigma(t).$$
Now,  by \eqref{E:eta}, for any $\eta > 0$ there exists $n_0$ such that for any $n \geq n_0$ we have 
$$\left| F\left( \frac{d(o, g_1 \dots g_n o) - n \ell}{\sqrt{n}} \right) - F\left( \frac{\beta_\xi(o, g_1 \dots g_n o) - n \ell}{\sqrt{n}} \right) \right| < \eta$$
on a set of $\overline{\nu}$-measure $\geq 1 - \epsilon$, hence 
$$\int_{\Omega \times \overline{X}^h}  F\left(  \frac{d(o, g_1\dots g_n o) - n \ell}{\sqrt{n}}  \right) d \overline{\nu}(\omega, \xi)  \to \int_{\mb R} F(t) \ d \mc N_\sigma(t).$$
Since the integrand does not depend on $\xi$, then we also have 
$$\int_\Omega  F\left(  \frac{d(o, g_1\dots g_n o) - n \ell}{\sqrt{n}}  \right) d P(\omega)  \to \int_{\mb R} F(t) \ d \mc N_\sigma(t),$$
where $P$ is the pushforward of $\overline{\nu}$ to $\Omega$.
Finally, since $\overline{\nu} = \frac{1}{R} \sum_w \mathbb{P}_w \otimes \nu_w$, the pushforward of $\overline{\nu}$ to $\Omega$ equals $P = \sum_w \frac{n_w}{R} \mathbb{P}_w$, where $n_w = \nu_w(M)$. Hence Lemma \ref{lem:better_work} implies that $P = \sum_w \pi_w \mathbb{P}_w = \mathbb{P}$, 
thus we also have
$$\int_\Omega  F\left(  \frac{d(o, g_1 \dots g_n o) - n \ell}{\sqrt{n}}  \right) d\mathbb{P}(\omega)  \to \int_{\mb R} F(t) \ d \mc N_\sigma(t)$$
as required.
\end{proof}

\section{Uniqueness of drift and variance}

Now suppose that $\Gamma$ is a semisimple graph structure. In particular, each maximal component $C_i$ of $\Gamma$ gives a primitive graph structure (without an initial vertex) on $G$ to which the results of the previous section (in particular, Theorem \ref{T:Markov}) applies. Hence for each maximal component $C_i$ of $\Gamma$, Theorem \ref{T:Markov} gives constants $\ell_i$ and $\sigma_i$ for the associated CLT.

In this section, we show that the CLTs for the recurrent components of $\Gamma$ are compatible in the sense that they have the same drift and variance. This is the primary place where we will use thickness and biautomaticity of $\Gamma$.

\begin{remark}
Our standing assumption until Section \ref{sec:almost_semi} is that $\Gamma$ is a semisimple graph structure on $G$. This implies that the transition matrix for each component of maximal growth is aperiodic. 
\end{remark}

\subsection{Uniformly bicontinuous functions}
Let us begin by introducing a class of functions that are well behaved under bounded perturbations in the group.

Let $\Omega^*$ be the set of finite length paths in $\Gamma$ starting at any vertex. 
Throughout this section, we fix a word metric $d_G$ on $G$.

\begin{definition}
A function $f \colon \Omega^* \to \mathbb{R}$ is \emph{uniformly bicontinuous} if for any finite set $B \subseteq G$ and any $\eta > 0$, 
there exists $N\ge 0$ such that if $\Vert g \Vert \geq N$ and 
$$b_1 \overline g b_2 = \overline h$$
in $G$ for some $b_1, b_2 \in B$,
then 
$$\left| f(g) - f (h) \right| < \eta.$$
\end{definition}

The following lemma gives the control we need on the length of paths in the graph structure.

\begin{lemma} \label{L:norm-control}
Suppose that the graph structure $\Gamma$ for $G$ is biautomatic. 
For any finite set $B \subseteq G$ there exists a constant $\mc B \ge 0$ such that 
if $g$ and $h$ are finite length paths in $\Gamma$, 
and $\overline{g} = b_1 \overline{h} b_2$ in $G$, then
$$| \Vert g \Vert - \Vert h \Vert | \leq \mc B.$$
\end{lemma}

\begin{proof}
Recalling that our graph structures are proper, we note that since there are only finitely many group elements of a given length, the function 
\[
n \mapsto \min\{d_G(1,\overline x) : \Vert x \Vert = n\}
\]
goes to infinity with $n$. Hence, we choose $\mc B$ sufficiently large so that any path $x$ in $\Gamma$ of length at least $\mc B$ satisfies $d_G(1,\overline x) \ge 2C +1$, where $C$ is as in Definition \ref{D:fellow-travel} for the given set $B$.

Now let $g,h \in \Omega^*$ be as given in the statement of the lemma and suppose that  $\Vert h \Vert = m \le n = \Vert g \Vert$. Since the graph structure of biautomatic, we have that both $d_G(\overline{g(n)}, b_1 \overline{h(n)})$ and $d_G(\overline{g(m)}, b_1 \overline{h(m)})$ are no more than $C$. But $h = h(m) = h(n)$ and so we see that 
$d_G(g(m), g(n)) \le 2C$. So if $p$ is the path in $\Gamma$ of length $n-m$ such that $g(m)\cdot p =g(n)$, then $d_G(1, \overline p) \le 2C$. Hence, $\Vert g \Vert - \Vert h\Vert = n-m = \Vert p \Vert \le \mc B$, by our above choice of $\mc B$. This completes the proof.
\end{proof}

We next introduce the primary functions of interest used throughout this section. Define the following functions on $\Omega^*$:
for any $\ell \in \mb R$, 

\[
\varphi(g) := \frac{d(o, g o) - \ell \Vert g \Vert}{\sqrt{\Vert g \Vert}}
\]
and 
\[
\psi(g) := \frac{d(o, g o)}{\Vert g \Vert}.
\]

Let $F \colon \mathbb{R} \to \mathbb{R}$ be a continuous, compactly supported function. Given a path $g = g_1 \dots g_n$ in the graph, denote 
$g_{[a, b]} = g_a g_{a+1} \dots g_b$
its subpath from position $a$ to position $b$.  

Finally, we define
\[
\mathcal{S}_n F(g) := \sum_{i = 0}^{\Vert g \Vert - n} F(\varphi(g_{[i+1, i+n]})).
\]

\begin{lemma} \label{L:uniform}
If the graph structure and biautomatic (and proper), then the functions $\psi$ and $\varphi$ defined above are uniformly bicontinuous. Moreover, 
for any continuous, compactly supported $F$. 
\[
 \frac{\mathcal{S}_n F(g)}{\Vert g \Vert}
\]
is also uniformly bicontinuous.
\end{lemma}

\begin{proof}
Suppose that $\overline{h} = b_1 \overline{g} b_2$ for some $b_1, b_2 \in B$. Then by Lemma \ref{L:norm-control}
$$|\Vert g \Vert - \Vert h \Vert | \leq \mathcal{B}$$
and by the triangle inequality
$$|d(o, g o) - d(o, h o)| \leq \mathcal{B}_1 $$
where $\mathcal{B}_1 := 2 \max_{b \in B} d(o, bo)$.
Finally, denote by $\mc L$ the Lipschitz constant so that 
$$d(o, g o) \leq \mc L \Vert g \Vert$$
for any $g \in \Omega^*$.

(1) By the above estimates, 
$$\left| \frac{d(o, h o)}{\Vert h \Vert} - \frac{d(o, g o)}{\Vert g \Vert} \right| \leq \frac{\mc L \mc B + \mc B_1}{\Vert g \Vert - \mc B},$$
and the right-hand side tends to $0$ as $\Vert g \Vert \to \infty$.

(2)
We can write
$$\left| \varphi(g) - \varphi(h) \right| = \frac{x}{\sqrt{n}} - \frac{x + y}{\sqrt{n + d}}$$
where $x = d(o, g o) - \ell \Vert g \Vert$, $y = d(o, h o) - \ell \Vert h \Vert- d(o,  g o) + \ell \Vert g \Vert$, 
$n = \Vert g \Vert $, and $d = \Vert h \Vert - \Vert g \Vert$.

Recall that by the above inequalities 
$$|d| \leq \mc B$$
hence also 
$$|y| \leq \mc B_1 + \ell \mc B$$ 
and 
$$|x| \leq (\mathcal{L}+\ell) \Vert g \Vert.$$
Thus, 
\begin{align*}
\left| \frac{x}{\sqrt{n}} - \frac{x + y}{\sqrt{n + d}} \right| & = \left| \frac{x(\sqrt{n+d}-\sqrt{n})}{\sqrt{n (n+d)}} - \frac{y}{\sqrt{n+d}} \right| \\
& \leq \frac{|x|}{n} \frac{n}{\sqrt{n(n+d)}} \left| \sqrt{n+d}-\sqrt{n} \right| + \frac{|y|}{\sqrt{n+d}} \\
& \leq (\mathcal{L}+\ell) \frac{n}{\sqrt{n(n- \mathcal{B})}} (\sqrt{n+ \mathcal{B}}-\sqrt{n}) + \frac{\mc B_1 + \ell \mc B}{\sqrt{n-\mathcal{B}}}
\end{align*}
and the right-hand side tends to $0$ uniformly in $n$.

(3) Fix $\epsilon > 0$. Since $F$ is uniformly continuous, let us pick $\delta > 0$ such that $|F(x) - F(y)| < \epsilon$ whenever $|x-y| < \delta$. 
By Definition \ref{D:fellow-travel}, there exists a finite set $B' \subseteq G$ such that if $\overline{h} = b_1 \overline{g} b_2$, then 
for any $i \leq \min \{ \Vert g \Vert, \Vert h \Vert \} - n$ there exist $b_3, b_4 \in B'$ such that
$$\overline{g_{1} \dots g_{i}} = b_1 \overline{h_{1} \dots h_{i}} b_3$$
$$\overline{g_{1} \dots g_{i+n}} = b_1 \overline{h_{1} \dots h_{i+n}} b_4$$
Thus, 
$$\overline{g_{i+1} \dots g_{i+n}}  = \left(\overline{g_{[1, i]}}\right)^{-1} \overline{g_{[1, i+n]}} = b_3^{-1} \overline{h_{i+1} \dots h_{i+n}} b_4.$$
Hence, since $\varphi$ is uniformly bicontinuous, there exists $N$ such that 
$$\left| \varphi(g_{i+1} \dots g_{i+n}) - \varphi(h_{i+1} \dots h_{i+n}) \right| < \delta$$
for any $n \geq N$. 
Thus, by the choice of $\delta$, 
$$\left| F(\varphi(g_{i+1} \dots g_{i+n})) - F(\varphi(h_{i+1} \dots h_{i+n})) \right| < \epsilon.$$
Since there are at most $ \Vert g \Vert$ terms you can compare and the additional terms (of which there are at most  $2 |\Vert g \Vert - \Vert h \Vert| \leq 2 \mathcal{B}$), are bounded by $\Vert F \Vert_\infty$, 
we have 
$$| \mathcal{S}_n F(g) - \mathcal{S}_n F(h) | \leq \epsilon \Vert g \Vert +  2 \mathcal{B} \Vert F \Vert_\infty.$$
Thus 
$$\left| \frac{\mathcal{S}_n F(g)}{\Vert g \Vert} - \frac{\mathcal{S}_n F(h)}{\Vert h \Vert} \right| \leq \frac{\mathcal{S}_n F(g)}{\Vert g \Vert} \cdot \frac{ \mathcal{B}}{\Vert g \Vert -  \mathcal{B}} + \frac{\epsilon \Vert g \Vert + 2\mathcal{B} \Vert F \Vert_\infty }{\Vert g \Vert - \mathcal{B}}$$
hence, noting that $| \mathcal{S}_n F(g) |\leq \Vert F \Vert_\infty \Vert g \Vert$, we obtain 
$$\limsup_{\Vert g \Vert \to \infty} \left| \frac{\mathcal{S}_n F(g)}{\Vert g \Vert} - \frac{\mathcal{S}_n F(h)}{\Vert h \Vert} \right| \leq \epsilon$$
and, since this is true for any $\epsilon$, the claim follows.
\end{proof}

\begin{remark}[Logarithmic perturbations] \label{rmk:log_change}
As a consequence of the proof that $\varphi$ is uniformly bicontinuous, we observe that for any $\eta > 0$ there is an
$N$ such that if $\Vert g \Vert \geq N$ then for any decomposition $g = g_0 g_1 g_2$ with $\Vert g_0 \Vert, \Vert g_2 \Vert \leq \log N$
we have 
$$\left| \varphi(g) - \varphi (g_1) \right| < \eta.$$
\end{remark}

The main reason why we introduce the bicontinuous functions is the following property. Recall that we denote $\mu_n^{(i)}$ the $n$th step 
distribution of the Markov chain associated to the maximal component $C_i$. 

\begin{lemma} \label{L:same-limit}
Suppose that the structure is thick and let $C_i$ be a maximal component. Let $f \colon \Omega^* \to \mathbb{R}$ be a uniformly bicontinuous function, and suppose that for some constant $a$
$$f(w_n) \to a$$ 
in probability with respect to the Markov chain measure on $C_i$. 
Then for any other maximal component $C_j$, we also have 
$$f(w_n) \to a$$ 
in probability with respect to the Markov chain measure on $C_j$. 
\end{lemma}

\begin{proof}
Let $v$ be a vertex of $C_i$, and let $C_j$ be another maximal component. By thickness, we have 
$$S_n^{(j)} \subseteq \bigcup_{b_1, b_2 \in B} \bigcup_{|k| \leq \mathcal{B}} b_1 (\overline{\Gamma_v \cap S_{n + k}}) b_2$$
where $S_n^{(j)}$ is the set of paths of length $n$ which entirely lie in $C_j$.
Hence, for any $\epsilon > 0$ there exists $n_0$ such that for all $n \geq n_0$
$$\#\{ g \in S_n^{(j)} \ : \   |f(g) - a | > \epsilon \} \leq \mathcal{B}^2 \sum_{|k| \leq \mathcal{B}} \#\{ g \in \Gamma_v \cap S_{n+k} \ : \  |f(g) - a | > \epsilon/2 \}.$$
Now, note that there exists $C > 0$ such that 
\begin{align} \label{eq:quasi_equiv}
C^{-1} \frac{\#(A \cap S_n^{(i)})}{\lambda^n} \leq \mu_n^{(i)}(A) \leq C \frac{\#(A \cap S_n^{(i)})}{\lambda^n}
\end{align}
for any $i$, any $n$ and any set $A$, hence, by noting that $\Gamma_v \cap S_{n+k} \subseteq S_{n+k}^{(i)}$, 
$$\mu_n^{(j)}( g  \ : \   |f(g) - a | > \epsilon ) \leq \mathcal{B}^2 C^2 \sum_{|k| \leq \mathcal{B}} \mu_{n+k}^{(i)}( g  \ : \  |f(g) - a | > \epsilon/2 ) $$
Now, since $f(w_n) \to a$ in probability with respect to $\mu^{(i)}_n$, the right-hand side tends to $0$, hence also the left-hand side does.
\end{proof}

\subsection{Uniqueness of drift}
We now show that all maximal components of $\Gamma$ determine the same drift.

\begin{lemma} \label{L:unique-drift}
If the structure is thick, then all $\ell_i$ for all maximal components $C_i$ are the same.
\end{lemma}

\begin{proof}
By the ergodic theorem, for each maximal component $C_i$ we have 
$$\lim_{n \to \infty} \frac{d(o, w_n o)}{n} = \ell_i$$
almost surely (and hence in probability) with respect to the Markov chain measure $\mu_n^{(i)}$. 
Since $\psi(g) = \frac{d(o, g o)}{\Vert g \Vert}$ is uniformly bicontinuous by Lemma \ref{L:uniform}, the claim then follows by Lemma \ref{L:same-limit}. 
\end{proof}

By the above lemma, we now define $\varphi$ using $\ell = \ell_i$ for any (equivalently all) $i$.

\subsection{Uniqueness of variance}
In a similar setting, Calegari--Fujiwara \cite{CF} use the notion of \emph{typical path} to show 
that all maximal components have the same variance. 

Here, we adapt this technique by using the uniform bicontinuity of the functions $\frac{\mathcal{S}_n F(g)}{\Vert g \Vert}$, plus thickness of our structure. 
An important difference is that here we use convergence in probability instead of almost sure convergence, as the first one can be ``transferred" from one 
component to another using thickness (see Lemma \ref{L:same-limit}).

\begin{lemma} \label{L:sigma-typical}
Let $C_i$ be a maximal component, and let $\sigma_i$ be the variance of the corresponding CLT. 
Then for any compactly supported, continuous function 
$F \colon \mathbb{R} \to \mathbb{R}$ and for any $n$, there exists a constant $\mathcal{E}_n^{(i)}(F)$ such that 
for any $\epsilon > 0$, 
$$ \lim_{m \to \infty} \mu^{(i)}_{m} \left(g \ : \  \left|  \frac{\mathcal{S}_n F(g)}{\Vert g \Vert}  - \mathcal{E}_n^{(i)}(F) \right| \geq \epsilon \right) \to 0$$
and moreover
$$\lim_{n \to \infty} \mathcal{E}_n^{(i)}(F) = \int_{\mathbb{R}} F(t) \ d \mc N_{\sigma_i}(t).$$
\end{lemma}

\begin{proof}
Recall that for the Markov chain, by the ergodic theorem the limit
$$\mathcal{E}_n^{(i)}(F) := \lim_{m \to \infty} \frac{1}{m} \sum_{i = 0}^{m-1} F(\varphi(g_{i+1}\dots g_{i+n})) = \int_G F(\varphi(g)) \ d\mu^{(i)}_{n}(g)$$
exists almost surely, hence also in probability. Moreover, since the law of $\varphi(w_{n})$ converges to $\mc N_{\sigma_i}$, we have 
$$\lim_{n \to \infty} \mathcal{E}_n^{(i)}(F) = \lim_{n \to \infty} \int_G F(\varphi(g)) \ d\mu^{(i)}_{n}(g) = \int F(t) \ d \mc N_{\sigma_i}(t).$$
Thus, for any $\epsilon > 0$, 
$$ \lim_{m \to \infty} \mu_{n+m}^{(i)} \left(g  \ : \  \left| \frac{1}{m} \sum_{i = 0}^{m-1} F(\varphi(g_{i+1}\dots g_{i+n}))  - \mathcal{E}_n^{(i)}(F) \right| \geq \epsilon \right) \to 0.$$
Since $\lim_{m \to \infty} \frac{\Vert g \Vert}{m} = \frac{n+m}{m} = 1$ in the above equation, the claim follows. 
\end{proof}

\begin{lemma} \label{L:unique-variance}
Suppose that $G$ has a thick graph structure. Then for any two maximal components $C_i, C_j$ we have $\sigma_i = \sigma_j$.
\end{lemma}

\begin{proof}
Fix $F$ continuous and compactly supported, and let $M_n F(g) := \frac{\mathcal{S}_n F(g)}{\Vert g \Vert}$. Fix a maximal component $C_i$. By Lemma \ref{L:sigma-typical}, for any $n$ the function $M_n F(w_m)$ converges in probability w.r.t. the Markov measure for $C_i$ as $m \to \infty$ to some constant $\mathcal{E}_n^{(i)}(F)$. Since $M_n F$ is uniformly bicontinuous by Lemma \ref{L:uniform}, we have by Lemma \ref{L:same-limit} that $M_n F(w_m)$ converges to the same constant in probability with respect to the Markov chain for $C_j$. Hence, $\mathcal{E}_n^{(i)}(F) = \mathcal{E}_n^{(j)}(F)$ for any $n$.
Thus,   
$$\int F(t) \ d \mc N_{\sigma_i}(t)  = \lim_{n \to \infty} \mathcal{E}_n^{(i)}(F) = \lim_{n \to \infty} \mathcal{E}_n^{(j)}(F) =  \int F(t) \ d \mc N_{\sigma_j}(t).$$
Since, this is true for any $F$, we must have $\sigma_i = \sigma_j$.
\end{proof}

\section{The semisimple case}

In this section, we prove our main theorem for semisimple graph structures. This is completed in Theorem \ref{T:main-ss}.

\subsection{Convergence to the Markov measure}
So far our work has been for maximal components of a semisimple graph structure. In this section we consider the whole graph structure, still in the semisimple case.

Let $\Gamma$ be a semisimple graph structure for $G$ with transition matrix $M$ of spectral radius $\lambda > 1$. Let $v_i$ be the vertices of the graph, 
and let $v_0$ be a vertex of large growth, which we take as the initial vertex. 
Then recall that $e_i^T M^n e_j$ is the number of paths of length $n$ from $v_i$ to $v_j$. 
Since $M$ is semisimple, the limit 
$$M_\infty := \lim_{n \to \infty} \frac{M^n}{\lambda^n}$$
exists. In particular, in keeping with notation at the beginning of Section \ref{sec:clt_primitve}, we define 
\begin{align*}
\rho_i := \lim_{n \to \infty} \frac{e_i^T M^n 1}{\lambda^n} \quad \text{and} \quad u_i := \lim_{n \to \infty} \frac{e_0^T M^n e_i}{\lambda^n}.
\end{align*}
By construction, $\rho = (\rho_i)$ satisfies $\rho = M_\infty 1$ and $M \rho = \lambda \rho$, while $u = (u_i)$ satisfies $u^T M = \lambda u^T$.
Finally, $\sum_i u_i = \rho_0$ and $\sum_i u_i \rho_i = \rho_0$. 

Note that vertices $v_i$ for which $\rho_i > 0$ and $u_i > 0$ are precisely vertices of components of maximal growth. The large growth vertices are those with $\rho_i >0$.

As before, we use a standard construction to define a measure $\mathbb{P}$ on the space $\Omega$ of infinite paths starting at any vertex of $\Gamma$. First define the initial distribution of the Markov chain to start at vertex $v_i$ with probability $\pi_i := \frac{u_i \rho_i}{\rho_0}$. 
Then assign an edge from $v_i$ to $v_j$ the probability 
$p_{ij} := \frac{\rho_j }{\lambda \rho_i}$.  Obviously, $\mathbb{P}$ is supported on paths that are entirely contained in components of maximal growth.
We denote as $\mathbb{P}_n$ the distribution of the $n$th step of the Markov chain.  

\begin{remark}
We remark that the induced measure on each maximal component $C$ of $\Gamma$ rescales to give the Markov measure on $C$ previously considered. This follows immediately from the construction.
\end{remark}

The following result relates the Markov measure on the semisimple graph structure to the counting measure.
For its statement, let $v_0$ be any vertex of large growth. For each $n$, consider the path given by selecting uniformly a path $\gamma$ starting at $v_0$ of length $n$, and take its subpath $\tilde{\gamma}$ from position $\log n$ to position $n - \log n$. Let $\tilde{\lambda}_n$ denote the distribution of $\tilde{\gamma}$.  

\begin{lemma} \label{L:counting-markov}
With notation as above, the total variation
$$\Vert \mb P_{n-2\log(n)}  - \tilde{\lambda}_{n}\Vert_{TV} \to 0$$
as $n \to \infty$.
\end{lemma}

\begin{proof}
Denote $n' := n - 2 \log n$. Let $\gamma$ be a path in the graph, starting at $v_i$ and ending at $v_j$. Then by definition the proportion of paths of length $n$, starting at $v_0$, that have $\gamma$ as ``middle subpath" of length $n'$ is 
$$\tilde{\lambda}_n(\gamma) = \frac{(e_0^T M^{\log n} e_i) (e_j^T M^{\log n} 1)}{e_0^T M^n 1}.$$
On the other hand, 
$$\mb P_{n'}(\gamma) = \left\{ \begin{array}{ll} \frac{\pi_i \rho_j}{\rho_i \lambda^{n'}} & \textup{If }v_i\textup{ has large growth} \\
0 & \textup{otherwise,}\end{array}\right.$$
which is nonzero if both $v_i$ and $v_j$ belong to a maximal component. 
In this case, 
\begin{align*} 
\frac{d\mb P_{n'}}{d \tilde{\lambda}_{n}}(\gamma) &= \frac{\lambda^{\log n}}{e_0^T M^{\log n}e_i}\cdot \frac{\lambda^{\log n}}{e_j^T M^{\log n} 1} \cdot \frac{e_0^T M^n 1}{\lambda^n} \cdot \frac{\pi_i \rho_j}{\rho_i}\\
&\longrightarrow \frac{1}{u_i} \cdot \frac{1}{\rho_j} \cdot \frac{\rho_0}{1} \cdot \frac{\pi_i \rho_j}{\rho_i} =1
\end{align*}
using that $\pi_i =\frac{u_i \rho_i}{\rho_0}$. Moreover, if $S_n^{i,j}$ denotes the set of paths of length $n'$ from $v_i$ to $v_j$, we have
\begin{align*}
\tilde{\lambda}_{n}(S_n^{i,j})
& =  \frac{(e_i^T M^{n'} e_j) (e_0^T M^{\log n} e_i) (e_j^T M^{\log n} 1)}{e_0^T M^n 1} \\
& \leq \frac{(e_i^T M^{n'} 1) (e_0^T M^{\log n} e_i) (e_j^T M^{\log n} 1)}{e_0^T M^n 1}  \\
& \to \frac{\rho_i u_i \rho_j}{\rho_0},
\end{align*}
hence such a probability tends to $0$ unless both $v_i$ and $v_j$ belong to a maximal component. 

Finally, if we denote as $\mathcal{L}_n$ the set of paths of length $n'$ which lie entirely in 
a maximal component, we have for any set $A$
\begin{align*}
\left| \mb P_{n'}(A)  - \tilde{\lambda}_{n}(A) \right| 
& \leq \sum_{x \in A \cap \mathcal{L}_n} \left| \frac{ \mb P_{n'}(x)}{\tilde{\lambda}_{n}(x)}\tilde{\lambda}_{n}(x) - \tilde{\lambda}_{n}(x) \right| + 
\tilde{\lambda}_{n}(A \setminus \mathcal{L}_n)  \\
& \leq \sup_{x \in \mathcal{L}_n} \left| \frac{\mb P_{n'}(x)}{\tilde{\lambda}_{n}(x)} - 1 \right| + \tilde{\lambda}_{n}(\mathcal{L}_n^c)
\end{align*}
and both terms tend to $0$ as $n \to \infty$, independently of $A$.
\end{proof}

\subsection{Central limit theorem for the counting measure in the semisimple case}

We are now ready to prove the following. For its statement, let $S_n$ denote the set of length $n$ paths beginning at the initial vertex $v_0$.

\begin{theorem}  \label{T:main-ss}
Let $\Gamma$ be a semisimple, thick, biautomatic graph structure for a nonelementary group $G$ of isometries of a $\delta$-hyperbolic space $(X, d)$, and let $o \in X$ be a base point.
Then there exists $\ell \geq 0$, $\sigma \geq 0$ such that for any $a < b$ we have
$$\lim_{n \to \infty} \frac{1}{\# S_n} \# \left\{ g \in S_n \ : \   \frac{d(o, g o) - n \ell}{\sqrt{n}}  \in [a, b] \right\} =  \int_a^b d \mc N_\sigma(t).$$
\end{theorem}

In the following proof and later on, we will use the notation $N_\sigma(x) := \int_{-\infty}^x d\mc N_\sigma(t)$.

\begin{remark} \label{R:one-comp}
Note that if, additionally, the graph structure $\Gamma$ is semisimple and has a unique maximal component, Theorem \ref{T:main-ss}
holds even without assuming that the structure is biautomatic.
\end{remark}

\begin{proof}
Let $C_1, \dots, C_k$ be the maximal components, and let $\mu_n^{(i)}$ be the $n$th step distribution for the Markov chain 
associated to that component, as in Section \ref{sec:clt_primitve}. Theorem \ref{T:Markov} shows a CLT for all such measures, and by Lemmas \ref{L:unique-drift} and \ref{L:unique-variance} all such measures have the same drift and variance, that we denote by $\ell, \sigma$. 

Now, since the starting probability $(\pi_i)$ in the above construction is nonzero precisely on the set of vertices which belong to a maximal component, 
there exist weights $c_i \geq 0$ with $\sum_i c_i = 1$ such that
$$\mb P_n = \sum_{i = 1}^k c_i \mu_n^{(i)}$$
for any $n$. Thus, for any $x \in \mb R$, 
\begin{equation} \label{E:MC-CLT}
\mb P_n (g \ : \ \varphi(g) \leq x ) = \sum_{i = 1}^k c_i \mu_n^{(i)} (g \ : \ \varphi(g) \leq x ) \to N_\sigma(x),
\end{equation}
where we recall that 
\[
\varphi(g) = \frac{d(o, go) - \ell \Vert g \Vert}{\sqrt{\Vert g \Vert}}.
\]

Now, we use that the counting measure can be approximated by the Markov chain measure.
If $g$ is a path of length $n$, we denote as $g = g_0 g_1 g_2$ where $g_0$ is the prefix of length $\log n$, 
$g_1$ is the middle part of length $n - 2 \log n$ and $g_2$ is the final part of length $\log n$.
By Remark \ref{rmk:log_change}, there exists $n_0$ such that 
\begin{equation} \label{E:log-change}
| \varphi(g) - \varphi(g_1) | \leq \epsilon
\end{equation}
for any $n \geq n_0$ and $g$ with $\Vert g \Vert = n$. 

Fix $x \in \mb R$ and $\epsilon > 0$. Then we have 
\begin{align*}
\lambda_n ( g \ : \ \varphi(g) \leq x ) & = \lambda_n ( g = g_0 g_1 g_2 \ : \ \varphi(g) \leq x )  \\
\intertext{and, by eq. \eqref{E:log-change}, for $n$ large }
& = \lambda_n ( g = g_0 g_1 g_2 \ : \ \varphi(g_1) \leq x + \epsilon ) \\
\intertext{then by definition of $\tilde{\lambda}_n$}
& = \tilde{\lambda}_n ( g_1 \ : \ \varphi(g_1) \leq x + \epsilon ) \\
\intertext{and by Lemma \ref{L:counting-markov}, for $n$ large,}
& = \mb P_{n - 2 \log n} ( g_1 \ : \ \varphi(g_1) \leq x + \epsilon ) + \epsilon. \\
\intertext{Hence, by eq. \eqref{E:MC-CLT} we obtain}
\limsup_{n \to \infty} \lambda_n ( g \ : \ \varphi(g) \leq x) & \leq N_\sigma(x + \epsilon) + \epsilon
\intertext{and, by taking $\epsilon$ smaller and smaller and using the continuity of $N_\sigma$,}
\limsup_{n \to \infty}  \lambda_n ( g \ : \ \varphi(g) \leq x) & \leq N_\sigma(x).
\end{align*}
The lower bound follows analogously.
\end{proof}

Indeed, the same proof shows the following stronger statement. Let $\lambda_n^{(i)}$ denote the counting measure on the set of 
paths of length $n$ starting at $v_i$. 

\begin{corollary}  \label{T:main-ss-every}
Let $\Gamma$ be a semisimple, thick, biautomatic graph structure for a nonelementary group $G$ of isometries of a $\delta$-hyperbolic space $(X, d)$, and let $o \in X$ be a base point.
Then there exists $\ell \geq 0$, $\sigma \geq 0$ such that for any vertex $v_i$ of large growth for $\Gamma$ and any $a < b$ we have
$$\lim_{n \to \infty} \lambda_n^{(i)} \left( g \ : \   \frac{d(o, g o) - \ell \Vert g \Vert}{\sqrt{\Vert g \Vert}}  \in [a, b] \right) =  \int_a^b d N\sigma(t).$$ 
\end{corollary}

\begin{proof}
Let us fix a vertex $v_\ell$ of large growth for $M$. Then we can define a Markov chain measure $\mb P^{(\ell)}$ on the space of infinite paths 
as follows. The transition probabilities will always be the same $p_{ij} = \frac{\rho_j}{\lambda^p \rho_i}$, while 
for each vertex $v_\ell$ one finds a different set of starting probabilities $\pi_i^{(\ell)}$ given by 
$$\pi_i^{(\ell)} := \frac{u_i^{(\ell)} \rho_i}{\rho_\ell}, \quad \text{where} \quad u_i^{(\ell)} := \lim_{n \to \infty} \frac{e_\ell^T M^{n} e_i}{\lambda^{n}}.$$ 
Just as before, there exist constants $c_i^{(\ell)} \geq 0$ such that $\sum_i c_i^{(\ell)} = 1$ and 
$$\mb P_n^{(\ell)} = \sum_{i =1}^k c_i^{(\ell)} \mu_n^{(i)}.$$
The proof then proceeds exactly as for Theorem \ref{T:main-ss}. 
\end{proof}

\section{The CLT for displacement and translation length} \label{sec:almost_semi}

Now suppose that $\Gamma$ is an almost semisimple graph structure for $G$ with transition matrix $M$. Then $M$ has some period $p\ge1$ such that $M^p$ is semisimple. We denote by $\Gamma^p$ the corresponding $p$ step graph structure on $G$. That is, $\Gamma^p$ is the graph with the same vertex set as $\Gamma$ and an edge joining $v_i$ to $v_j$ for each directed path from $v_i$ to $v_j$ of length $p$, whose label is the word in $G$ spelled by the corresponding path. The transition matrix for $\Gamma^p$ is $M^p$, hence $\Gamma$ is a semisimple graph structure for $G$.

Since the previous results require this structure to be thick, we need the following lemma.

\begin{lemma} \label{lem:passing_to_p}
The following properties pass to the $p$ step graph structure:
\begin{itemize}
\item If $v$ is a large growth vertex of $\Gamma$, then its also a large growth vertex of $\Gamma^p$.
\item If $\Gamma$ is a thick structure, then $\Gamma^p$ is also thick. 
\item If $\Gamma$ is biautomatic, then so is $\Gamma^{p}$.
\end{itemize}
\end{lemma}

\begin{proof}
The first statement holds because any path from $v$ that ends in a component of maximal growth can be extended to a path whose length is a multiple of $p$ by adding on a path in that component of length less than $p$.

Now suppose that $\Gamma$ is thick.
Let $v$ be a vertex in a maximal component of $\Gamma^p$. Then $v$ is also a vertex in a maximal component of $\Gamma$. Let $\Gamma_{v, p}$ be the semigroup 
of loops based at $v$ of lengths multiple of $p$. Consider the semigroup homomorphism 
$$f \colon \Gamma_v \to \mathbb{N} \to \mathbb{N}/p\mathbb{N}$$
given by taking the length and reducing it mod $p$. Clearly, the image of $f$ is a subsemigroup of $\mathbb{N}/p\mathbb{N}$, which is a finite group,
hence the image is also a group. Let $\gamma_i, \dots, \gamma_k \subseteq \Gamma_v$ be a set of representatives for each remainder class in the image of $f$. Now, let $\gamma \in \Gamma_v$. Then $\Vert \gamma \Vert$ belongs to the image of $f$, hence there exists $\gamma_i$ (the representative of the inverse modulo $p$), such that $\gamma \gamma_i$ has length multiple of $p$, hence it belongs to $\Gamma_{v, p}$. Hence, by setting 
$B'$ the set $\{ \gamma_i^{-1} \ : \ 1 \leq i \leq k \}$, we have $\Gamma_v \subseteq \Gamma_{v, p} B'$ in the group. 
Since $\Gamma$ is thick, there exists $B''$ such that $G = B'' \Gamma_v B''$, hence also $G = B'' \Gamma_{v, p} B' B''$, 
hence $\Gamma^p$ is also thick. 

Finally, note that any path $g$ in $\Gamma^p$ of length $k$ can be naturally thought of as a path $g_{\dagger}$ in $\Gamma$ of length $pk$ such that for all $i\ge 0$: $\overline{g_{\dagger}(pi)} = \overline g(i)$. Hence, if $B$ and $C\ge 0$ are fixed as in the statement of biautomaticity of $\Gamma$ and $g,h$ are paths in $\Gamma^p$ with $\overline g = b_1\overline h b_2$ for $b_1,b_2 \in B$, then we also have that $\overline {g_{\dagger}} = b_1\overline {h_{\dagger}} b_2$. Then biautomaticiy of $\Gamma$ implies that
\[
d_G (\overline{g_{\dagger}(i)}, b_1 \overline{h_{\dagger}(i)} ) \le C,
\]
for all $i \ge 0$. Hence, restricting to multiples of $p$ completes the proof.
 \end{proof}

Now, let us consider the semisimple matrix $M^p$. Note that irreducible components of $M^p$ may be proper subsets of irreducible components of $M$. 
Given a vertex $v_i$, let us denote by $\lambda_{k}^{(i)}$ the counting measure on paths starting at $v_i$ of length $k$ for $\Gamma$. Note that if $k = np$, then this also counts paths of length $n$ in $\Gamma^p$ starting at $v_i$.

By applying Theorem \ref{T:main-ss-every} to $\Gamma^p$, we immediately obtain:

\begin{corollary} \label{T:CLT-p-count}
Let $\Gamma$ be a thick, biautomatic structure of period $p$ for a nonelementary group $G$ of isometries of a $\delta$-hyperbolic space $(X, d)$, and let $o \in X$ be a base point. Then there exists $\ell, \sigma$ such that the following holds. 
For any vertex $v_i$ of large growth for $\Gamma$ and for any $x$, we have 
$$\lambda_{pn}^{(i)} \left(g \ : \ \frac{d(o, go) - \ell \Vert g \Vert}{\sqrt{\Vert g \Vert}} \leq x \right) \to \int_{-\infty}^x d \mc N_{\sigma}(t)$$
as $n \to \infty$.
\end{corollary}

We are now ready to prove the following. Recall that $S_n$ denotes the set of length $n$ paths beginning at the initial vertex $v_0$.

\begin{theorem}  \label{T:main-almostss}
Let $\Gamma$ be a thick, biautomatic graph structure for a nonelementary group $G$ of isometries of a $\delta$-hyperbolic space $(X, d)$, and let $o \in X$ be a base point.
Then there exists $\ell \geq 0$, $\sigma \geq 0$ such that for any $a < b$ we have
$$\lim_{n \to \infty} \frac{1}{\# S_n} \# \left\{ g \in S_n \ : \   \frac{d(o, g o) -  \ell n}{\sqrt{n}}  \in [a, b] \right\} = \int_a^b d \mc N_\sigma(t).$$
\end{theorem}

\begin{proof}

Let $v_0$ be the initial vertex, let $S_{n}$ be the set of paths of length $n$ based at $v_0$, and let $\lambda_n$ be the uniform measure on $S_n$. 

Let us fix $0 \leq r \leq p-1$. Then we can write the counting measure on $S_{pn+r}$, starting at the initial vertex $v_0$, 
by first picking randomly a path $g_0$ of length $r$ from $v_0$ with a certain probability $\mu$, and then picking 
a random path starting at $v_i = t(g_0)$ with respect to the counting measure on the set of paths of length $n$ starting at $v_i$.

To compute $\mu$, let us consider a path $g_0$ of length $r$ starting at $v_0$ and ending at $v_i$. Then, if $v_i$ is of large growth for $\Gamma^{p}$, 
$$\frac{\#\{ \textup{paths from }v_i\textup{ of length }pn \}}{\#\{ \textup{paths from }v_0\textup{ of length }pn + r\}} = \frac{e_i M^{pn} 1}{e_0 M^{pn+r} 1} \to \frac{e_i M_{\infty} 1}{e_0 M^r M_\infty 1}.$$
Thus, we define 
$$\mu(g_0) :=  \frac{e_i M_{\infty} 1}{e_0 M^r M_\infty 1}.$$
Note that $\mu(g_0) = 0$ if the end vertex of $g_0$ has small growth 
and moreover
$$\sum_{\Vert g_0 \Vert = r} \mu(g_0) = \sum_i \mu(g_0) \# \{ g_0 \in S_r \ :  \ t(g_0) = v_i \} = \sum_i e_0 M^r e_i \frac{e_i M_{\infty} 1}{e_0 M^r M_\infty 1}  = 1.$$

Let $\lambda'_{pn+r}$ be the measure on $S_{pn + r}$ given by first taking randomly a path $g_0$ of length $r$ from $v_0$ 
with distribution $\mu$ and then taking uniformly a path of length $pn$ starting from $t(g_1)$. 

Now we show that the CLT holds for $\lambda'_{pn+r}$. Let $\ell, \sigma$ be given by Theorem \ref{T:CLT-p-count}, and let $\varphi(g) := \frac{d(o, go) - \ell \Vert g\Vert}{\sqrt{ \Vert g \Vert}}$.
By Theorem \ref{T:CLT-p-count}, for any vertex $v_i$ of large growth, we have 
$$\lambda^{(i)}_{pn} ( g :  \ \varphi(g) \leq x ) \to N_{\sigma}(x).$$
Then if $g = g_0 g_1$, and $t(g_0)$ denotes the (index of the) end vertex of $g_0$, 
\begin{align*}
\lambda'_{pn+r} ( g  \ : \ \varphi(g) \leq x ) &= \sum_{g_0 \in S_r} \mu(g_0) \lambda^{(t(g_0))}_{pn} ( g_1 :  \ \varphi(g_0 g_1) \leq x ) \\
&\longrightarrow \sum_{g_0} \mu(g_0) N_{\sigma}(x) = N_\sigma(x),
\end{align*}
where we used that $\varphi$ is uniformly bicontinuous as in the proof of Theorem \ref{T:main-ss}.

Now we prove that
$$\Vert \lambda'_{pn+r} - \lambda_{pn+r} \Vert_{TV} \to 0$$
as $n \to \infty$. 
Indeed, if $\gamma = g_0 g_1$ is a path from $v_0$ of length $pn+r$ and $g_0$ is its prefix of length $r$ ending at a vertex $v_i$ of large growth, then 
$$\frac{ \lambda'_{pn+r}(\gamma) }{\lambda_{pn+r}(\gamma)} = \frac{\mu(g_0) \cdot \frac{1}{e_i M^{pn} 1}}{\frac{1}{e_0 M^{pn+r} 1} } \to 1.$$
On the other hand, if the end vertex of $g_0$ is of small growth, then $\lambda'_{pn+r}(g) = 0$, 
and also 
$$\lambda_{pn+r}( g = g_0 g_1 \ : \  g_0 \textup{ ends at a small growth vertex} ) \to 0$$
as $n \to \infty$.
Now, let $A_x := \{ g  \ : \ \varphi(g) \leq x \}$ and $L_r$ be the set of paths starting at $v_0$ whose prefix of length $r$ ends in a vertex of large growth. Then 
\begin{align*}
\lambda_{pn+r}(g  \ : \ \varphi(g) \leq x ) & = \lambda_{pn+r} (g \in L_r \ : \ \varphi(g) \leq x ) + \lambda_{pn+r} (g \notin L_r \ : \ \varphi(g) \leq x )  \\
& = \frac{\lambda_{pn+r}(A_x \cap L_r) }{\lambda'_{pn+r}(A_x \cap L_r)} \lambda'_{pn+r}(A_x \cap L_r) + \lambda_{pn+r}(A_x \setminus L_r) \\
& \to 1 \cdot N_\sigma(x) + 0 = N_\sigma(x).
\end{align*}
We have thus obtained a CLT for $\lambda_{pn+r}$, for any $0 \leq r \leq p-1$, always with the same $\ell, \sigma$. 
Since there are only finitely many values $r$, the claim follows.
\end{proof}

\subsection{A CLT for translation length}

We now prove a more general version of our second main result, Theorem \ref{T:main} (2). 

\begin{theorem} \label{T:CLT-tau}
Let $\Gamma$ be a thick, biautomatic graph structure for a nonelementary group $G$ of isometries of a $\delta$-hyperbolic space $(X, d)$, 
let $o \in X$ be a base point, and let $\ell,\sigma$ be as in Theorem \ref{T:main-almostss}. 
Then for any $a < b$ we have
$$\lim_{n \to \infty} \frac{1}{\# S_n} \# \left\{ g \in S_n \ : \   \frac{\tau(g) -  \ell n}{\sqrt{n}}  \in [a, b] \right\} =  \int_a^b d \mc N_\sigma(t).$$
\end{theorem}

\begin{proof}
Let us recall that the translation length of an isometry $g$ of a $\delta$-hyperbolic space can be computed by (see e.g. \cite[Proposition 5.8]{MT})
\begin{equation} \label{E:tau}
\tau(g) = d(o, go) - 2(go, g^{-1}o)_o + O(\delta)
\end{equation}
where $O(\delta)$ is a constant which only depends on the hyperbolicity constant of $X$. 
Now, by choosing $f(n) = \epsilon \sqrt{n}$ in  \cite[Proposition 5.8]{GTT2}, for any $\epsilon$ we have 
$$\lambda_n(g \ : \ (go, g^{-1} o)_o \leq \epsilon \sqrt{n} ) \to 1$$
as $n \to \infty$. The claim then follows by combining this statement and the statement of Theorem \ref{T:main-almostss} into formula \eqref{E:tau}.
\end{proof}

\subsection{Zero variance}

We finally complete our main theorem by characterizing the case where $\sigma = 0$. First, we give a general criterion.

\begin{proposition}\label{P:zero-sigma_gen}
In the hypotheses of Theorem \ref{T:main-almostss} we have $\sigma =0$ if and only if there is $C \ge0$ such that for all finite length paths $g$ in $\Gamma$,
\[
|d(o, go) - \ell  \Vert g \Vert |  \le C.
\]
\end{proposition}

We note that the proposition implies that $\sigma >0$ whenever the action $G \curvearrowright X$ is nonproper.

\begin{proof}
Suppose that $\sigma = 0$ for the CLT for the counting measure. Then by our previous discussion, we have 
$\sigma = 0$ also for the Markov chain on any maximal components. 
Then by Theorem \ref{T:suspend}, we also have $\sigma = 0$ for the random walk on the loop semigroup driven by $\check{\mu}_v$. 
Hence, as in \cite[Proof of Theorem 4.7 (b)]{BQ-hyperbolic}, for any $n$
$$\frac{1}{n} \int (\eta_0(g, \xi) )^2  d \check{\mu}_v^{*n}(g) d \nu_v(\xi) =  0$$
where $\eta_0$ is the centering of $\eta$. This implies
$$\eta_0(g, \xi) = 0$$
for any $g \in \Gamma_v^{-1}$ and $\nu_v$-a.e. $\xi \in \overline X^h$. 
Thus, since $|\eta - \eta_0| \leq 2 \Vert \psi \Vert_\infty$ is bounded, we have 
$$ |\beta_\xi(o, g^{-1} o) - \ell \Vert g \Vert | = |\eta(g, \xi)| \leq 2 \Vert \psi \Vert_\infty$$
hence by \cite[Corollary 2.3]{Horbez} there exists a constant $C$ for which
$$|d(o, go) - \ell \Vert g \Vert | \leq C$$
for any $g$ in the support of $\check{\mu}_v^{*n}(g)$.

Hence, by thickness we have for any $g \in \Omega^*$ there exists $b_1, b_2 \in B$ and $h \in \Gamma_v^{-1}$ such that 
$$\overline{h} = b_1 \overline{g} b_2,$$ 
thus by Lemma \ref{L:norm-control} and the triangle inequality
$$|d(o, go) - \ell \Vert g \Vert | \leq |d(o, ho) - \ell \Vert h \Vert| + \mc B_1 + \ell \mc B$$
thus there exists a constant $C'$ such that 
$$|d(o, go) - \ell \Vert g \Vert | \leq C'$$
for any $g \in \Omega^*$. This completes the proof.
\end{proof}

We conclude with a corollary that applies when the graph structure is geodesic. For the action $G \curvearrowright X$, denote the translation length of $h$ by $\tau_X(h)$. We use the notation $\tau_G(h)$ to denote the translation length of $h$ with respect to the word metric $d_G$ induced by the graph structure $\Gamma$:
\[
\tau_G(h) = \lim_{n\to \infty}\frac{1}{n}d_G(1, h^n).
\]

\begin{corollary} \label{C:zero-sigma}
Suppose that $\Gamma$ is a thick bicombing of $G$.
If $\sigma = 0$ in the CLT, then for all $h \in G$
$$\tau_X(h) = \ell \: \tau_G(h),$$
where $\ell$ is the corresponding drift.
\end{corollary}

\begin{proof}
Let $g_n$ be a path in $\Gamma$ representing $h^n$ for $h\in G$. 
That is $h^n = \overline g_n$.
Since the structure is geodesic, $||g_n|| = d_G(1,h^n)$. Applying Proposition \ref{P:zero-sigma_gen}, we get that 
\[
|d(o, h^no) - \ell d_G(1,h^n)|  = O(1).
\]
The corollary follows after dividing by $n$ and taking a limit.
\end{proof}

\section{Applications}\label{sec:Apps}
The main theorem of Section \ref{sec:intro} now follows easily from the results in Section \ref{sec:almost_semi}.

\begin{proof}[Proof of Theorem \ref{T:main}]
Since $G$ has a thick bicombing with respect to $S$,  the length $\Vert g \Vert$ of a path in the graph equals the word length with respect to $S$ of its evaluation $\overline{g} \in G$, and the sphere of radius $n$ in the Cayley graph of $G$ is in bijection with the set of paths of length $n$ in the graph. Then (1) follows immediately from Theorem \ref{T:main-almostss}, (2) follows from \ref{T:CLT-tau} and (3) from Corollary \ref{C:zero-sigma}.
\end{proof}

We now give proofs of the applications in the introduction. We first recall some examples of groups 
which admit thick bicombings; for further details, see also \cite{GTT2}.
 
\begin{lemma} \label{L:bicombing}
The following groups admit thick bicomings:
\begin{enumerate}
\item
A (word) hyperbolic group $G$ admits a thick bicombing with respect to any generating set. 
\item
If $G$ is relatively hyperbolic with virtually abelian peripheral subgroups, then every finite generating set $S'$ can be extended to a 
finite generating set $S$ for $G$ which admits a thick bicombing. 
\item
If $G$ is a right-angled Artin group or right-angled Coxeter group that does not decompose as a product and $S$ is the vertex generating set, then $G$ admits a 
thick combing for $S$ whose graph structure has only one maximal component, which is aperiodic. 
\end{enumerate}
\end{lemma}

\begin{proof}
(1) By \cite{Cannon}, a hyperbolic $G$ has a bicombing with respect to any generating set. Such a structure is thick by 
\cite[Lemma 4.6]{GMM}.

(2) 
By \cite[Corollary 1.9]{AC}, the generating set $S'$ of $G$ can be enlarged to a generating set $S$, so that the pair $(G,S)$ admits a 
geodesic graph structure. By \cite[Theorem 5.2.7]{HRR}, this can be turned into a bicombing for the same generating set $S$.

Hence, it remains to show that this bicombing is thick. Yang \cite{Yan} proves that any relatively hyperbolic group has the \emph{growth quasitightness} property (see \cite[Definition 1.2]{GTT2}, inspired by \cite{AL}) 
with respect to any finite generating set. Since growth quasitightness implies thickness by \cite[Proposition 7.2]{GTT2}, the proof is complete.

(3) In \cite[Corollary 10.4]{GTT2}, building on Hermiller--Meier \cite{HM}, we proved that the language of lexicographically first geodesics in the vertex generators is parameterized by a thick graph structure. In fact, the graph structure we construct has only one maximal component, which is aperiodic. 
\end{proof}

\begin{proof}[Proof of Theorem \ref{th:length}]
Note that $\pi_1(M)$ is hyperbolic relative to its parabolic subgroups, which are virtually abelian since $M$ has constant curvature. Hence, by Lemma \ref{L:bicombing} (2) 
the given generating set $S'$ can be enlarged to a finite generating set $S$ that is associated to a thick bicombing on $\pi_1(M)$. 
The theorem then follows from Theorem \ref{T:main}. Finally, $\sigma > 0$ by (3) since the length spectrum is not arithmetic (\cite{GR}, \cite{Kim}). 
\end{proof}

\begin{proof}[Proof of Theorem \ref{th:3d}]
In the case where $M$ has no rank $2$ cusps, we have that $\pi_1(M)$ is hyperbolic. Indeed, by the Tameness Theorem (\cite{CG}, \cite{agol-tame}), $M$ is the interior of a compact manifold $\overline M$, which by assumption does not have tori as boundary components. Then Thurston's Hyperbolization Theorem (see \cite{kapovich-book}), $\overline M$ admits a convex cocompact hyperbolic structure on its interior. Hence, $\pi_1(M)$ is hyperbolic. The result now follows from Lemma \ref{L:bicombing} and Theorem \ref{T:main}.

For the moreover statement, the argument above gives that $\overline M$ admits a geometrically finite hyperbolic structure. Hence, $\pi_1(M)$ is hyperbolic relative to its rank $2$ parabolic subgroups, which are virtually $\Z \times \Z$. The proof then proceeds as in Theorem \ref{th:length}
\end{proof}

\begin{proof}[Proof of Theorem \ref{th:intersection}]
First, since $\pi_1(M)$ is word hyperbolic, by Lemma \ref{L:bicombing} (1) it has a thick bicombing with respect to any generating set.

Second, let $T =T_\Sigma$ be the dual tree associated to $\Sigma \subset M$. For details of this standard construction and the properties we 
need, see \cite[Section 1.4]{Shalen}. Alternatively, $T$ is the Bass--Serre tree associated to the splitting of $\pi_1(M)$ induced by $\Sigma$. Since $\Sigma$ is not fiber-like, $T$ is not the real line, and since the quotient $\mc G$ of the action $\pi_1(M) \curvearrowright T$ is compact (it is the underlying graph of the associated graph-of-groups), the action is nonelementary. 

Finally, the intersection number $i(\gamma, \Sigma)$ equals the translation length of $\gamma$ with respect to the action $\pi_1(M) \curvearrowright T$. To see this, note that the translation length of $\gamma$ for this action is equal to the number of edges $\#_e\gamma$ crossed by the shortest representative of $\gamma$ in $\mc G$. If we embed $\mc G$ in $M$ dual to $\Sigma$, this shows that $i(\gamma, \Sigma) \le \#_e\gamma$. For the opposite inequality, recall that there is a retraction $r \colon M \to \mc G$ mapping each component of $\Sigma$ to the midpoint of some edge. Thus by taking a representative of $\gamma$ intersecting $\Sigma$ minimally, considering its image under the retraction, and homotoping it off edges that it does not fully cross, we obtain that $\#_e\gamma \le i(\gamma, \Sigma).$ Hence, $i(\gamma,\Sigma) = \ell(\gamma)$ for the action on $T$. 

We now obtain the CLT by applying Theorem \ref{T:main} to this action. If $\sigma = 0$, then Theorem \ref{T:main} (3) implies that the action $\pi_1(M) \curvearrowright T$ is proper. However, this is impossible since only virtually free groups admit proper actions on trees. 
\end{proof}

For the following application, let us assume $G$ is a hyperbolic group, let $\partial G$ be its Gromov boundary, 
and let $d$ be a metric on $G$.  
We define the \emph{growth rate} of the metric $d$ as 
$$v := \limsup_{n \to \infty} \frac{1}{n} \log \# \big \{g \in G \ : \ d(1, g) \leq n \big \}$$
and for each $s \geq v$ let us consider the measure on $G \cup \partial G$:
$$\nu_s := \frac{\sum_{g \in G} e^{-s d(1, g)} \delta_{g}}{ \sum_{g \in G} e^{-s d(1, g)}}.$$
Then any limit point of $(\nu_s)$ as $s \to v$ is supported on $\partial G$ and 
is called a \emph{Patterson--Sullivan (PS) measure}. By Coornaert \cite{Coornaert}, any two limit measures are absolutely continuous with respect to each other, with bounded Radon--Nikodym derivative, so the Patterson-Sullivan \emph{measure class} is well-defined.

\begin{proof}[Proof of Theorem \ref{th:homo-PS}]
Since $G$ is word hyperbolic, it has a thick bicombing by Lemma \ref{L:bicombing} (1). 
The first statement then follows immediately from Theorem \ref{T:main}, by considering the action of $G$ on the Cayley graph of $G'$.

For the moreover statement, if $\sigma=0$, Theorem \ref{T:main} (3) implies that 
\[
|\Vert \phi(g) \Vert_{S'}-  \ell \Vert g \Vert_{S}|
\]
is bounded independently of $g\in G$, hence $\phi$ has finite kernel. 

Now, consider the factorization $G \overset{\pi}{\rightarrow} \overline G : = \frac{G}{\ker \phi}  \overset{\overline{\phi}}{\rightarrow} G'$, 
and define $\overline{S} := \pi(S)$. Then the Cayley graph of $\overline G$ carries the two metrics 
$$d_1(g, h) := \Vert h^{-1} g \Vert_{\overline{S}} \qquad d_2(g, h) := \Vert \overline{\phi}(h^{-1}g) \Vert_{S'}$$ 
and they satisfy 
\begin{equation} \label{E:same-class}
|d_1(g, h) - \ell d_2(g, h) | \leq C
\end{equation}
for any $g, h \in \overline{G}$. 
Now, by \cite[Theorem 2]{Furman}, eq. \eqref{E:same-class} holds if and only if the Bowen--Margulis measures on the double boundary $\partial \overline G \times \partial \overline G$  associated to $d_1$, $d_2$ are the same, which by \cite[Proposition 1]{Furman} holds if and only if  the Patterson--Sullivan measure classes 
for $d_1$, $d_2$ on $\partial \overline{G}$ are the same. 
Finally, if $\phi$ has finite kernel, there exist $C > 0$ for which 
$$|d_1(\pi(g), \pi(h)) - d_S(g, h) | \leq C$$
for any $g, h \in G$. Hence, the PS measure class for $(G, d_S)$ on $\partial G$ pushes forward to the PS measure class for $(\phi(G), d_{S'})$ if and only if $\sigma = 0$. 
\end{proof}

\begin{proof}[Proof of Theorem \ref{th:raag}]
By Lemma \ref{L:bicombing} (3), a right-angled Artin or Coxeter groups has a graph structure with respect to the vertex generating set, 
which is semisimple with only one maximal component. Hence, the CLT follows from Theorem \ref{T:main-ss} (see Remark \ref{R:one-comp}).
To complete the proof, we note that $\#_v(g)$ is equal to the displacement of $g$ with respect to the action of $G$ on the Bass--Serre tree for the hyperplane associated to $v$. The details are similar to those of Theorem \ref{th:intersection}.
\end{proof}

\end{document}